\documentclass{amsart}
\usepackage{mathrsfs,amsmath}

\title[Crystalline representations with coefficients]
{Crystalline representations of $G_{\Q_{p^a}}$ with coefficients}

\author{Hui June Zhu}
\address{Department of Mathematics, The State University of New York at
Buffalo, Buffalo, New York 14260. USA
}

\date{}

\subjclass[2000]{11, 14.}

\usepackage{amsfonts}
\usepackage{amsmath}
\usepackage{amssymb}
\usepackage{amsthm}

\newtheorem{theorem}{Theorem}[section]
\newtheorem*{theorem-a}{Theorem A}
\newtheorem*{theorem-a'}{Theorem A'}
\newtheorem*{theorem-b}{Theorem B}
\newtheorem*{theorem-c}{Theorem C}
\newtheorem*{theorem-d}{Theorem D}
\newtheorem*{theorem-e}{Theorem E}
\newtheorem*{corollary-f}{Corollary F}
\newtheorem{proposition}[theorem]{Proposition}
\newtheorem{lemma}[theorem]{Lemma}
\newtheorem{corollary}[theorem]{Corollary}
\theoremstyle{remark}
\newtheorem{remark}[theorem]{Remark}

\newtheorem{acknowledgments}{Acknowledgments}

\newtheorem{conventions}{Conventions}

\newcommand{\E}{\mathbf{E}}
\newcommand{\D}{\mathbf{D}}
\newcommand{\bcris}{\mathbf{B}_{\mathrm{cris}}}

\newcommand{\dcris}{\mathbf{D}_{\mathrm{cris}}}

\newcommand{\vcris}{\mathbf{V}_{\mathrm{cris}}}

\newcommand{\MF}[1]{\mathbf{MF}^{\mathrm{#1}}}

\newcommand{\akplus}{\mathbf{A}_K^+}
\newcommand{\bkplus}{\mathbf{B}_K^+}
\newcommand{\aeplus}{\mathbf{A}_E^+}
\newcommand{\beplus}{\mathbf{B}_E^+}
\newcommand{\B}{\mathbf{B}}
\newcommand{\A}{\mathbf{A}}

\newcommand{\m}{\mathfrak{m}}
\newcommand{\hG}{\hat{G}}
\newcommand{\hP}{\hat{P}}
\newcommand{\hA}{\hat{A}}
\newcommand{\hM}{\hat{M}}
\renewcommand{\k}{\mathbf{k}}
\newcommand{\vk}{{\vec{k}}}
\newcommand{\0}{{\vec{0}}}
\newcommand{\1}{{\vec{1}}}

\newcommand{\vu}{\vec{u}}

\newcommand{\vv}{\vec{v}}

\newcommand{\valpha}{\vec{\alpha}}
\newcommand{\vX}{\vec{X}}
\newcommand{\N}{\mathbf{N}}

\newcommand{\fn}{\mathfrak{n}}

\newcommand{\Fil}{\mathrm{Fil}}

\newcommand{\HP}{\mathrm{HP}}
\newcommand{\NP}{\mathrm{NP}}
\newcommand{\Mat}{\mathrm{Mat}}
\newcommand{\diag}{\mathrm{Diag}}
\newcommand{\ord}{\mathrm{ord}}
\newcommand{\rank}{\mathrm{rank}}
\newcommand{\GL}{\mathrm{GL}}
\newcommand{\Gal}{\mathrm{Gal}}
\newcommand{\tr}{\mathrm{tr}}
\newcommand{\cris}{\mathrm{cris}}

\newcommand{\Id}{\mathrm{Id}}
\newcommand{\Hom}{\mathrm{Hom}}
\newcommand{\ind}{\mathrm{ind}}

\newcommand{\Rep}{\mathrm{Rep}}
\newcommand{\MaxSlope}{\mathrm{MaxSlope}}

\newcommand{\dist}{\mathrm{dist}}

\newcommand{\Z}{\mathbb{Z}}
\newcommand{\Zp}{{\mathbb{Z}_p}}
\newcommand{\C}{\mathbb{C}}
\newcommand{\F}{\mathbb{F}}
\newcommand{\Fp}{\mathbb{F}_p}
\newcommand{\Q}{\mathbb{Q}}

\newcommand{\Cp}{{\mathbb{C}_p}}
\newcommand{\Qp}{{\mathbb{Q}_p}}

\newcommand{\cA}{\mathscr{A}}
\newcommand{\cB}{\mathscr{B}}
\newcommand{\cO}{\mathscr{O}}
\newcommand{\cP}{\mathscr{P}}
\newcommand{\cS}{\mathscr{S}}
\newcommand{\cQ}{\mathscr{Q}}
\newcommand{\cR}{\mathscr{R}}
\newcommand{\fG}{\mathfrak{G}}
\newcommand{\fW}{\mathfrak{W}}
\newcommand{\fI}{\mathfrak{I}}

\renewcommand{\hat}{\widehat}
\renewcommand{\tilde}{\widetilde}
\renewcommand{\bar}{\overline}
\newcommand{\fk}{\lceil\bar{k}_1\rceil}
\newcommand{\ink}{\lfloor\bar{k}_1\rfloor}

\newcommand{\pscal}[1]{\langle #1 \rangle}

\begin{document}

\begin{abstract}
This paper studies crystalline representations of $G_{\Q_{p^a}}$
with coefficients of any dimension, where
$\Q_{p^a}$ is the unramified extension of $\Qp$ of degree $a$.
We prove a theorem of Fontaine-Laffaille type
when $\sigma$-invariant Hodge-Tate weight $<p-1$, 
which establishes the bijection between Galois stable lattices
in crystalline representations and strongly divisible $\varphi$-lattice.
In generalizing Breuil's work, we classify all reducible and irreducible crystalline representations of $G_{\Q_{p^a}}$ of dimensional $2$, then describe their mod $p$ reductions. We generalize some results (of Deligne, Fontaine-Serre, and Edixhoven)
to representations arising from Hilbert modular forms
when $\sigma$-invariant Hodge-Tate weight $<p-1$.
\end{abstract}

\maketitle

\tableofcontents

\section{Introduction}
\label{S:1}

Let $a\geq 1$ and let $K:=\Q_{p^a}$ be the unique unramified extension of
$\Qp$ of degree $a$. Let $G_K=\Gal(\bar{K}/K)$ be the Galois group of
the algebraic closure $\bar{K}$ of $K$ over $K$. This paper studies integral
theory concerning crystalline $E$-linear representations
of the Galois group $G_K$ over some coefficient field $E$.
For ease of notation we assume throughout this paper that $E$ contains $\Q_{p^a}$.

Let $\dcris^*$ be the equivalent functor
from the category
$\Rep_{\cris/\Qp}(G_K)$
of crystalline $\Qp$-linear representation of $G_K$ to
the category $\MF{ad}(K)$
of weakly admissible filtered $\varphi$-modules over $K$, defined
by $\dcris^*(V) = (\bcris\otimes_\Qp V^*)^{G_K}$ where $V^*=\Hom(V,K)$ is
the dual of $V$ (see \cite{Fo88a}\cite{Fo88b}\cite{CF00}).
Let $\Cp$ be the completion of $\bar\Q_p$.
Any crystalline representation $V$ is Hodge-Tate and hence
has Hodge-Tate weight, which we shall use
Hodge polygon to describe in this paper:
For any $d$-dimensional crystalline representation $V$,
an ordered $d$-tuple $\HP(V):=(k_1,\ldots,k_d)$ with
integers $k_1\geq \ldots\geq k_d$ is
Hodge polygon of $V$ if
$\dim (\Cp(-k_i)\otimes V)^{G_K}\neq 0$ and
each $k_i$ has multiplicity equal to $m_i:=\dim (\Cp(-k_i)\otimes V)$ in $\HP(V)$.
Write $D=\dcris^*(V)$, an equivalent definition of
$\HP(V)$ for $V$ in terms of its corresponding filtered $\varphi$-modules $D$ is
that:
$\Fil^{k_i}(D)\supsetneq \Fil^{k_i+1}(D)$ for every slope $k_i$ in $\HP(V)$
and $m_i=\dim_K(\Fil^{k_i}(D)/\Fil^{k_i+1}(D))$. For this reason
we shall use $\HP(V)$ and $\HP(D)$ interchangely for
any crystalline representations $V$.
We assume $V$ is effective, that is,
$k_i\geq 0$ for all $i$. We define $\MaxSlope(\HP(V)):=k_1-k_d$.
When $k_d=0$, which is the case for almost all cases
we consider, $\MaxSlope(\HP(V))=k_1$ the maximal slope of the Hodge polygon.

Let $T$ be any Galois stable lattice (over the $p$-adic integers $\Zp$) in $V$.
Laffaille has shown that $V$ is crystalline if and only if $\dcris^*(V)$ contains a strongly divisible lattice (\cite{Laf80}).
If $\MaxSlope(\HP(V)) < p-1$, then Fontaine-Laffaille theorem says that
if $V\in \Rep_{\cris/\Qp}(G_K)$ then
$\dcris^*(T)$ is a strongly divisible lattice in $\dcris^*(V)$ for every
Galois stable lattice
$T$ in $V$.
Namely, $\dcris^*$ gives a bijection from a Galois stable lattice $T$
in $V \in \Rep_{\cris/\Qp}(G_K)$ to a strongly divisible lattice $L$
in $D=\dcris^*(V)$. Fontaine-Laffaille theorem (\cite{FL82}) has played a prominent  role in computing mod $p$ reduction of the crystalline representation.

First of all, we give a structural construction of crystalline
representations of $G_K$ of any dimension, and if one wishes, with any prescribed
Hodge polygon. Let $\k=(k_1,\ldots,k_d)$ be in $\Z^d_{\geq 0}$ with $k_1\geq \ldots k_d\geq 0$ be a Hodge polygon, and let
$\Rep_{\cris/\Qp}^\k (G_K)$ be the subcategory of $\Rep_{\cris/\Qp}(G_K)$
where each representation has $\HP(V)=\k$.
Let $\cO_K$ be the ring of integers of $K$, and let the $\sigma$-linear {\em parabalic equivalence} $\sim_{\sigma,\k}$
be an equivalence relation on the set $\GL_d(\cO_K)$ which we shall
define in Section \ref{S:2.2}.

\begin{theorem-a}
There is a bijection
$$\Theta: \GL_d(\cO_K)/\sim_{\sigma,\k} \longrightarrow \Rep_{\cris/\Qp}^\k (G_K)$$
which sends every $A\in \GL_d(\cO_K)$ upon $\sigma$-linear parabolic equivalence
with respect to $\k$ to a crystalline representation $V$
whose $\dcris^*(V)$ contains a strongly divisible lattice
$L$ such that $\Mat(\varphi|_L)= A\cdot\diag(p^{k_1},\ldots,p^{k_d})$,
with respect to a basis adapted to its filtration.
\end{theorem-a}

Breuil and M\'ezard generalized Fontaine's $\dcris^*$ functor to
those representations with coefficient field $E$ that is no longer necessarily equal to
$\Qp$ (see \cite[Section 4]{BM02}), namely, $\dcris^*: \Rep_{\cris/E}(G_K)\longrightarrow \MF{ad}(K\otimes_\Qp E)$.
In this paper we extend Breuil-M\'ezard's result one step further to strongly divisible lattices, and obtain Fontaine-Laffaille type theorem in $\Rep_{\cris/E}(G_K)$. To do so, we shall define
{\em embedded Hodge polygon} $\HP_{\Z/a\Z}(V)=\k:=(k_{j,1},\ldots,k_{j,a-1})_{j\in\Z/a\Z}$
in Section \ref{S:3.1}. For any embedded Hodge polygon $\k$, let $\Rep_{\cris/E}^\k(G_K)$ be the subcategory consisting of all $V$ with
$\HP_{\Z/a\Z}(V)=\k$. Let $\cO_E$ be the ring of integers in $E$.
Let $\sim_{\k}$ be an equivalence relation on $\GL_d(\cO_E)^a$
which we shall define in Section \ref{S:3.1}.

\begin{theorem-a'}
Then there is a bijection
\begin{eqnarray*}
\Theta: \GL_d(\cO_E)^a/\sim_{\k} &\longrightarrow & \Rep_{\cris/E}^\k(G_K).
\end{eqnarray*}
\end{theorem-a'}

Let $\D$ be Fontaine's equivalence functor
from the category $\Rep_{/E}(G_K)$ of $E$-linear representations of
$G_K$ to the category of \'etale $(\varphi,\Gamma)$-modules. This functor respects
Galois stable lattices in $V$ of $\Rep_{/E}(G_K)$ and respects torsion (see \cite{FO91}).
For every crystalline representation $V$ in $\Rep_{\cris/E}(G_K)\subset \Rep_{/E}(G_K)$,
there is an integral Wach module in
the \'etale $(\varphi,\Gamma)$-module $\D(V)$ (see \cite{Wa96} and \cite{BE04}).
In fact, there is fully faithful functor $\N$ from $\Rep_{\cris/E}(G_K)$ to
the category of Wach modules that is bijection from Galois stable lattice to integral Wach modules (see \cite{BE04}) for $\Rep_{\cris/\Qp}(G_K)$.
In this paper, we study representations $V\in \Rep_{\cris/E}(G_K)$,
define $\sigma$-invariant Hodge polygon of $V$ by
$\HP^\sigma(V):=
(\frac{1}{a}\sum_{j\in\Z/a\Z}k_{j,1},\ldots,\frac{1}{a}\sum_{j\in\Z/a\Z}k_{j,d})$.
Let $\MaxSlope(\HP^\sigma(V)):=\frac{1}{a}(\sum_{j\in\Z/a\Z}(k_{j,1}-k_{j,d}))$.
We have the following Fontaine-Laffaille type theorem:

\begin{theorem-b}
Let $V \in \Rep_{\cris/E}(G_K)$ with $\MaxSlope(\HP^\sigma(V))<p-1$.
Let $T$ be a Galois stable lattice in $V$, then $\dcris^*(T)$ is a strongly divisible lattice in $\dcris^*(V)$.
Conversely, every strongly divisible lattice $L$ in $\dcris^*(V)$
is of the form $\dcris^*(T)$ for a Galois stable lattice $T$ in $V$.
\end{theorem-b}

It is desirable to explicitly construct integral Wach modules, especially
when $\MaxSlope(\HP^\sigma(V))\geq p-1$, hence compliments
Fontaine-Laffaille type theorem of strongly divisible lattices in
filtered $\varphi$-modules in $\MF{ad}(K\otimes_\Qp E)$.
Let $\m_E$ be the maximal ideal in $\cO_E$.
For any $V\in \Rep_{/E}(G_K)$ its {\em mod $p$ reduction} means
$T/\m_E T$ for any Galois stable lattice $T$ in $V$.
Its semisimplification does not depend on the choice of $T$
and we denote it by $\bar{V}$ in this paper.
Its $p$-power reductions are similarly defined.
To use integral Wach modules to calculate mod $p$-power reduction
of a crystalline representation, it  is essential
that we show $\N$ is a $p$-adic continuous map in a sense that
if two $d$-dimensional crystalline representations $V_1$ and $V_2$
in $\Rep_{\cris/E}^\k(G_K)$ are close in
a $p$-adic metric (which we shall define in
Remark \ref{R:metric} for $\Rep_{\cris/\Qp}^\k(G_K)$ and
Remark \ref{R:metric-2} for $\Rep_{\cris/E}^\k(G_K)$), then
their Wach modules are $p$-adically close.
Let $\lfloor \cdot\rfloor$ and $\lceil\cdot\rceil$
be the floor and ceiling of a real number.

\begin{theorem-c}
Write $\bar{k}_1=\MaxSlope(\HP^\sigma(V))$.
Given $V,V'\in\Rep_{\cris/E}^\k(G_K)$
such that $N$ is an integral Wach module of $V$.
Suppose
$\dist(V,V')\leq p^{-( i+ \lfloor \lfloor\bar{k}_1\rfloor p/(p-1)^2\rfloor)}$,
then there exists an integral Wach module $N'$ of $V'$ such that
$N\equiv N'\bmod \m_E^i.$
\end{theorem-c}

Finally, we are concerned with $2$-dimensional crystalline representations
of $G_K$ exclusively.
Recall that a $2$-dimensional representation
$V$ of $G_{\Qp}$ is irreducible if and only if the slope of its Newton polygon
has slope $>0$. Breuil has completely classified them in \cite{Br02}.
Theorem D below generalizes this to $G_K$.

\begin{theorem-d}
Let $V$ be any $2$-dimensional crystalline representation
with embedded Hodge polygon $\k=(k_j,0)_{j\in\Z/a\Z}$.
Then $V$ is (absolutely) irreducible if and only
if the Newton polygon of $V$ has no slope equal to $\sum_{j\in J}{k_j}$ for any subset $J$ of $\Z/a\Z$
(if $J$ is empty set then the sum is $0$).
\end{theorem-d}

Let $\omega_a$ be a fundamental character of the inertial group
$I_K$ of $G_K$, and let $\omega_{2a}$ be a fundamental character of $I_{\Q_{p^{2a}}}$
such that $\omega_{2a}^{p^a+1}=\omega_a$.
For any $1\leq h \leq p^a-1$ let $\ind(\omega_{2a}^h)$ be the (unique) irreducible representation $\bar\rho$ of $G_K$ with $\det(\bar\rho)=\omega_a^h$ and
$\bar\rho|_{I_K}\cong \omega_{2a}^h\oplus \omega_{2a}^{p^a h}$.
We compute mod $p$ reduction of crystalline representations
in Theorem E below. In this theorem, we adapt the notation in line with  Buzzard-Diamond-Jarvis \cite{BDJ08} and Gee \cite{Ge08b}
by using the {\em Hodge-Tate weight} $\vk=(k_0,\ldots,k_{a-1})$
instead of our embedded Hodge polygon $\k=(k_j,0)_{j\in\Z/a\Z}$.
We hope no confusion arises.

\begin{theorem-e}
Let $V$ be a 2-dimensional $E$-linear crystalline representation of
$G_K$ with Hodge-Tate weight $\vk=(k_0,\ldots,k_{a-1}) \in \Z_{\geq 0}^a$
such that $0\leq k_j\leq p-1$ for all $j$.
Let $\rho=\bar{V}$ be the semisimplification of mod $p$ reduction of $V$.
\begin{enumerate}
\item[(i)]
Suppose $\rho$ is reducible, then
$$
\rho|_{I_K} \cong \left(
           \begin{array}{cc}
           \omega_a^{\sum_{j\in \Z/a\Z}k_jp^j } & \star \\
           0     & 1 \\
        \end{array}
      \right) \otimes \bar\eta
$$
for some character $\bar\eta$ that extends to $G_K$.

\item[(ii)]
Suppose $\rho$ is irreducible and $\vk\neq \0$. Then
$$\rho\cong \ind(\omega_{2a}^{\sum_{j\in\Z/a\Z}k_jp^j})\otimes \bar\eta$$ for
 some character $\bar\eta$; or equivalently,
$$
\rho|_{I_K}\cong
\left(
           \begin{array}{cc}
           \omega_{2a}^{\sum_{j\in \Z/a\Z}k_jp^j } & 0 \\
           0     & \omega_{2a}^{p^a \sum_{j\in \Z/a\Z}k_jp^j} \\
        \end{array}
      \right) \otimes \bar\eta
$$
for some character $\bar\eta$ that extends to $G_K$.
\end{enumerate}
\end{theorem-e}

We remark that Breuil has classified in \cite{Br07} all reduction and irreducible
$\bar\F_p$-linear representations of $G_K$ in the above two forms in
Theorem E.

Theorem E has applications to Hilbert modular forms.
Let $F$ be any totally real field extension over $\Q$ where
$p$ is unramified, and we assume the localization of $F$ at the prime over $p$
is equal to $K$. Let $\fn$ be a nonzero ideal of
the ring of integers $\cO_F$ of $F$. Fix once and for all
embeddings $\bar\Q\hookrightarrow \C$ and $\bar\Q\hookrightarrow \bar\Q_p$.
Let $f$ be any Hilbert cuspidal eigenform of weight
$\vk=(k_0,\ldots,k_{a-1})$ for $0\leq k_j\leq p-1$, and of level $\fn$.
Let $\rho_f: G_K\rightarrow \GL_2(\bar\Q_p)$ be the Galois representation
associated to $f$
constructed by Rogawski-Tunnell \cite{RT83}, Ohta \cite{Oh84} and Carayol
\cite{Car86}, completed by Taylor \cite{Tay89} and Jarvis \cite{Jar97}.
Let $\bar\rho_f: G_F\longrightarrow \GL_2(\bar\F_p)$ be the semisimplification
of the reduction mod $p$ of $\rho_f$. We write $\bar\rho_{f,p}:=\bar\rho_f|_{G_K}$.
Suppose $p$ is coprime to $\fn$ then
$\rho_f|_{G_K}$ is crystalline by \cite{Br99} and
\cite{Liu07} (see also \cite{BE04}).
Then the following corollary follows immediately from Theorem E, and it
generalizes previous results of Deligne, Fontaine-Serre, and Edixhoven
on representations arising from modular forms (see \cite{Ed98} for references)
to Hilbert modular forms.

\begin{corollary-f}
Let $p>2$. Let $f$ be a Hilbert cuspidal eigenform of weight $\vk$ with $0\leq k_j\leq p-1$,
and of level $\fn$ coprime to $p$.
\begin{enumerate}
\item[(i)]
Suppose $\bar\rho_{f,p}$ is reducible, then
$$
\bar\rho_{f,p}|_{I_K} \cong \left(
           \begin{array}{cc}
           \omega_a^{\sum_{j\in \Z/a\Z}k_jp^j } & \star \\
           0     & 1 \\
        \end{array}
      \right) \otimes \bar\eta
$$
for some character $\bar\eta$ that extends to $G_K$.

\item[(ii)]
Suppose $\bar\rho_{f,p}$ is irreducible, we have for $\vk\neq \0$
$$
\bar\rho_{f,p}|_{I_K}\cong
\left(
           \begin{array}{cc}
           \omega_{2a}^{\sum_{j\in \Z/a\Z}k_jp^j } & 0 \\
           0     & \omega_{2a}^{p^a \sum_{j\in \Z/a\Z}k_jp^j} \\
        \end{array}
      \right) \otimes \bar\eta
$$
for some character $\bar\eta$ that extends to $G_K$.
If $\vk=\0$ then we have
$\bar\rho_{f,p}\cong \ind(\omega_{2a}^{p^a-1})\otimes\bar\eta$ for some
character $\bar\eta$.
\end{enumerate}
\end{corollary-f}

This paper is organized as follows.
Fundamental notions and preparations are recalled or defined in
Section 2, we prove Theorem A after Proposition \ref{P:parabolic}.
In Section 3 we prove Theorem A' for $\Rep_{\cris/E}(G_K)$ in
Theorem \ref{T:A-analog}. This section prepares some foundations on strongly divisible $\varphi$-lattices.
In Section 4 we study embedded integral Wach modules
in Fontaine's \'etale $(\varphi,\Gamma)$-modules,
and prove one of our main Theorem B in Theorem \ref{T:B}.
We prove $p$-adically continuity of Wach modules in Theorem C as part of
Corollary \ref{C:continuity}.
Finally in Section 5 we mainly focus on 2-dimensional case:
by proving irreducibility
in Theorem D that generalizes Breuil's classification \cite{Br02};
and computing mod $p$ reductions in Theorem E for $\bar{k}_1\leq p-1$.
At the end of the paper we demonstrate explicitly construction
of some families of integral Wach modules corresponding to
$2$-dimensional crystalline $E$-linear representations of $G_K$,
which extends some constructions done in \cite{BLZ04} for $K=\Qp$.

\begin{conventions}
Throughout this paper $p$ is a prime number.
We assume $K=\Q_{p^a}$ the unramified extension of $\Qp$ of degree $a$
everywhere (except declare otherwise in Section 4).
Let $E$ be an extension of $\Qp$
and without loss of generality we assume $K$ lies in $E$.
Let $\cO_K,\cO_E$ be their rings of integers, respectively.
Let $\m_E$ be the maximal ideals of $\cO_E$.
For any vector $\vv=(v_i)_i$ with $v_i\in E$,
let $\ord_p(\vv)=\min_j(\ord_p v_j)$.

Let $\pi$ be a variable. We write
$\akplus=\cO_K[[\pi]]$, $\aeplus=\cO_E[[\pi]]$,
$\bkplus=\akplus[\frac{1}{p}]$, Let $\A_K$
be the ring of power series $\sum_{i=-\infty}^{\infty}a_i \pi^i$
such that $a_i \in \cO_K$ and
$a_i\rightarrow 0$ as $i\rightarrow -\infty$. Let $\B_K$
be its field of fractions.
Let $p$-adic order $\ord_p(\cdot)$ on $\cO_E[[\pi]]$ (or $\cO_K[[\pi]]$) be defined as the minimal $p$-adic order of polynomial coefficients.
We have $\A_E/p\A_E\cong \cO_E/\m_E ((\pi))$.
For any $c \neq 0$ let $\cR_{c,E}$ be the ring of all power series
$\sum_{s=0}^{\infty}a_s \pi^s$ in $E[[\pi]]$ such that $\ord_p a_s
\geq  - \frac{s}{c}$. For instance, $\cR_{\infty,E}=\aeplus$.

Let $I_K$ be the inertial group of $G_K$.
Let $\omega_a$ be a (Serre's) fundamental character
of $I_K$ given by composing $I_K\rightarrow \F_{p^a}^*$ (via
local class field theory) with an embedding $\tau_0: \F_{p^a}\rightarrow \bar\F_p$.
Let $\sigma$ be the absolute Frobenius on $\bar\F_p$,
then we may identify even element $\tau_j:=\tau_0\circ \sigma^j$ in $\Gal(\F_{p^a}/\Fp)$ with $j$ in $\Z/a\Z$. Indeed, we identify the subindex set $\Z/a\Z$
with $\Gal(K/\Qp)\cong\Gal(\F_{p^a}/\F_p)$ in this paper.
\end{conventions}

\begin{acknowledgments}
We thank Laurent Berger and Kiran Kedlaya for pointing out some
key references to us during our research, with double thanks to Laurent Berger for helpful comments on an early version of this paper.
\end{acknowledgments}

\section{Weakly admissible filtered $\varphi$-modules}
\label{S:2}

\subsection{Preliminary algebra}
\label{S:2.1}

A {\em filtered module} $D$ over a ring $\cO$ is a
$\cO$-modules with decreasing filtration (i.e.,
$\Fil^iD\supseteq \Fil^{i+1}D$ for $i\in \Z$) that is
separated (i.e., $\Fil^iD = 0$ for $i\gg 0$)
and exhaustive (i.e., $\Fil^i D = D$ for $i\ll 0$).
We say the filtration is {\em saturated} if
$\Fil^i D$ are saturated submodules.
Let $\cO$ be a discrete valuation ring with
field of fractions $Q$, if $\cO$ is a ring endowed with $\varphi$-action
that is an automorphism on $D\otimes_\cO Q$ such that
$\varphi(cv)=\sigma(c)\varphi(v)$
for $c\in Q$ and $v\in D$, where $\sigma$ is the absolute Frobenius of $Q$
(i.e., the lifting of Frobenius $x\mapsto x^p$ on a residue field of $\cO$)
then $D$ is a {\em filtered $\varphi$-module}.

A filtered $\varphi$-module over $K$ is just a filtered $K$-vector space
with an automorphism $\varphi$ such that
$\varphi(cv)=\sigma(c)\varphi(v)$ for the absolute Frobenius $\sigma$ on $K$.
The category of filtered $\varphi$-modules over $K$ is denoted by
$\MF{}(K)$. We write $(D,\Fil)$ for the filtered $K$-module $D$.
Morphisms between filtered $K$-modules
are {\em strict}, i.e., for any $f: (D,\Fil)\rightarrow (D',\Fil)$
we have $f(\Fil^i D ) = f(D)\cap \Fil^i D'$) unless otherwise declared.

We use Hodge polygon to describe the filtration data on $(D,\Fil)$
consisting of Hodge-Tate weights.
The {\em Hodge polygon} of a filtered module $D$ of dimension $d$ over $K$ is
given by a $d$-tuple $\HP(D/K)=(k_1,\ldots,k_d)=:\k$ with $k_1\geq
\cdots \geq k_d$ such that $\Fil^{k_i} D \supsetneq
\Fil^{k_i+1}D$, and the multiplicity of each $k_i$ is equal to
$\dim_K(\Fil^{k_i}D/\Fil^{k_i+1}D)$.
We always assume $k_d\geq 0$ (i.e., $D$ is effective) in this paper.
Write $r_i:=\dim_K \Fil^{k_i} D$. Then we have
$r_1\leq \ldots \leq r_d =d$ with the multiplicity of each value
in the same pattern as that for $\HP(D/K)$. Note that each Hodge
polygon slope $k_i$ has horizontal length equal to $r_i-r_{i-1}$
(by setting $r_0:=0$).
For any $\alpha=\frac{r}{s}\in\Q$ let $D_\alpha$ be the $K$-subvector space
generated by those $w\in D$ with $\varphi^s(w)=p^r w$. Then
$D=\oplus_{\alpha\in\Q} D_\alpha$.
The {\em Newton polygon} of $D$ is given by
a $d$-tuple $\NP(D/K)=(\beta_1,\ldots,\beta_d)$ where $\beta_j\in \Q$
are of multiplicity $\dim_K D_{\beta_j}$, and $\beta_1\geq \ldots\beta_d$.
A filtered $\varphi$-module $D$ over $K$ is {\em
weakly admissible} if for every subobject $D'$
of $D$ we have $\HP(D'/K) \succ \NP(D'/K)$ (where $\succ$ means `lies below')
and the endpoints for $\HP(D/K)$ and $\NP(D/K)$ coincide.
Let $t_N(D/K)=\sum_{\alpha\in\Q}\alpha\cdot \dim_KD_\alpha$ and
$t_H(D/K)=\sum_{i\in\Z} i\cdot\dim_K(\Fil^iD/\Fil^{i+1}D)$. Then
the weakly admissibility is equivalent to
that $t_H(D'/K)\leq t_N(D'/K), t_H(D/K)=t_N(D/K)$ for all sub-object $D'$ in $D$.
Let $\MF{ad}(K)$ be the subcategory of $\MF{}(K)$ consisting of
weakly admissible objects with strict morphisms.
For any Hodge polygon $\k=(k_1,\ldots,k_d)$,
let $\MF{ad,\k}(K)$ be subcategory of $\MF{ad}(K)$ consisting of
those with $\HP(V)=\k$.

A {\em $\varphi$-lattice} $L$ in $D$ with induced filtration is a
free $\cO_K$-submodule of $D$ such that
$$
L\otimes_{\cO_K}K = D,\quad
\varphi(L)\subseteq L, \quad
\Fil^i L = \Fil^i D \cap L.
$$
For the rest of this section all filtration on a lattice $L$ in $D$
are induced from $D$. If $L =
\pscal{e_1,\ldots,e_d}$ is a basis over $\cO_K$ such that $\Fil^i
L = \pscal{e_1,\ldots,e_{\rank(\Fil^i L)}}$ for every $i$, then it
is called {\em a basis of $L$ adapted to the filtration}. A basis
of $D$ over $K$ adapted to the filtration is a basis adapted to
the filtration for some lattices in $D$. The
existence of such basis is due to the following proposition.
(Remark: Fontaine-Rapoport \cite{FR} called $D=\pscal{e_i}_i$ a basis
adapted to the filtration if $\Fil^r D = \pscal{e_i}_{r_i\geq r}$
where  $r_i$ is the maximal number $r$ such that $e_i\in \Fil^r
D$. Our definition is slightly stronger.)

\begin{proposition}\label{P:basis}
(1) If $(L,\Fil)$ is a
saturated filtered $\cO$-module free of rank $d$ for
a principal ideal domain $\cO$,
then it contains a basis adapted to the filtration.

(2) Let $(D,\Fil)$ be any filtered $K$-module of dimension $d$.
Let $L$ be a $\cO_K$-lattice in $D$ with $\Fil^i L = \Fil^i D\cap L$.
Then there exists a basis $\pscal{e_1,\ldots,e_d}$ of $L$ over $\cO_K$
that is adapted to the filtration.
\end{proposition}
\begin{proof}
Let $d$ be the rank of $L$ over $\cO$.  It suffices to show that
if $L=L_0 \supseteq L_1 \supseteq \cdots \supseteq L_n \supsetneq 0$ for
free saturated $\cO_K$-submodules $L_i$'s of rank $r_i$, then
there exists a basis $\pscal{e_1,\ldots,e_d}$ of $L$ such that
$L_i=\pscal{e_1,\ldots,e_{r_i}}$.
We shall induce on $d$.
If $d=1$ then this claim obviously holds.
Pick a basis for the smallest filtration $N:=L_n = \pscal{e_1,\cdots,e_{r_n}}$
over $\cO_K$.
By taking quotient we get a filtration
$L/N \supseteq L_1/N \supseteq \cdots \supseteq L_{n-1}/N \supsetneq 0$
of free $\cO_K$-modules. It is saturated since
$L_i/L_{i+1}\cong (L_i/N)/(L_{i+1}/N)$ is torsion-free over $\cO_K$.
By induction $L/N = \pscal{\bar{e}_{r_n+1},\ldots,\bar{e}_d}$ of rank $d-r_n<d$
satisfies that
$L_i/N = \pscal{\bar{e}_{r_n+1},\ldots,\bar{e}_{r_n+r_i}}$ for every $i$.
Let $e_i\in L $ be a lift of $\bar{e}_i\in L/N$, then
we have that $L_i = \pscal{e_1,\ldots,e_{r_n},e_{r_n+1},\ldots,e_{r_n+r_i}}$
for every $i$. This proves part (1).

(2)
Clearly $\Fil^i L = \Fil^i D \cap L$ is a full lattice in $\Fil^i D$, i.e.,
it spans $\Fil^iD$ over $K$.
One notices that $\Fil^{i+1} L$ is a saturated submodule of $\Fil^{i}L$,
namely $\Fil^iL/\Fil^{i+1}L$ is torsion-free (or zero). Indeed, if we have
$r\in\cO_K$ and $x\in \Fil^i L$ such that
$rx\in \Fil^{i+1}L$, then $x\in \Fil^{i+1}D$ and hence $x\in \Fil^{i+1}L$
because filtration on $L$ is induced from that on $D$. Therefore part (2) follows from applying part (1) on $(L,\Fil)$ over $\cO_K$.
\end{proof}

\begin{lemma}\label{L:parabolic}
Let $D=\pscal{e_1,\ldots,e_d}/K$
be a filtered module with basis adapted to its
filtration and let $\HP(D/K)=(k_1,\ldots,k_d)=\k$.
Let $r_i=\dim_K \Fil^{k_i}D$.

(1) The automorphism group of $(D,\Fil)$ is isomorphic to the
subgroup $\cP_{\k,K}$ defined below in $\GL_d(K)$ by sending
every automorphism to
its matrix with respect to the given basis $\pscal{e_1,\ldots,e_d}$:
$$
\cP_{\k,K}:=
\left(
  \begin{array}{cccc}
    \GL_{r_1}(K) & * & * & *\\
    0 & \GL_{r_2-r_1}(K) & * & * \\
    \vdots &&\ddots & * \\
    0 & \cdots & 0 & \GL_{d-r_{d-1}}(K)  \\
  \end{array}
\right)
$$
where $*$ are arbitrary entries in $K$.

(2) Let $L$ be a $\cO_K$-lattice in $(D,\Fil)$ with induced
filtration. The automorphism group of $(L,\Fil)$ is
\begin{eqnarray*}
\cP_\k
&=&\GL_d(\cO_K)\cap \cP_{\k,K} \\
&=& \left(
  \begin{array}{cccc}
    \GL_{r_1}(\cO_K) & * & * & *\\
    0 & \GL_{r_2-r_1}(\cO_K) & * & * \\
    \vdots &&\ddots & * \\
    0 & \cdots & 0 & \GL_{d-r_{d-1}}(\cO_K)  \\
  \end{array}
\right) \end{eqnarray*}
where $*$ are arbitrary entries in $\cO_K$.
\end{lemma}
\begin{proof}
(1) Let $\theta$ be an automorphism of $(D,\Fil)$. Write $\theta
(e_1,\ldots,e_d) = (e_1,\ldots,e_d)\Theta$ for $\Theta\in
\GL_d(K)$. Then for every $i$ we have
$\theta(e_1,\ldots,e_{r_i})=(e_1,\ldots,e_{r_i})A_{r_i}$ for some
$A_{r_i}\in \GL_{r_i}(K)$. Thus
$$\Theta\cdot \diag(\Id_{r_i},0,\ldots,0) = \diag(A_{r_i},0,\ldots,0),$$
where $\Id_{r_i}$ is the identity matrix in $\GL_{r_i}(K)$.
This proves part (1).

(2)
Similar argument with $\Theta\in\GL_d(\cO_K)$ and $A_{r_i}\in\GL_{r_i}(\cO_K)$.
\end{proof}

We call the above groups $\cP_{\k,K}$ and $\cP_\k$ the
{\em parabolic group} and {\em integral parabolic group with respect to $\k$},
respectively. When there is no confusion
in context we also call $\cP_\k$ the parabolic group of $\k$.
For example, if the filtration data on $D$ is maximal in the sense that
$\dim \Fil^i D = \dim \Fil^{i+1} D + 1$ for every $i$, then
$\cP_{\k}$ is the Borel subgroup of $\GL_d(\cO_K)$. Namely,
it is the upper triangular matrix with $\cO_K^*$ on the diagonal.
Our definitions here are inspired by work of
Fontaine-Rapoport \cite{FR}.

\begin{proposition}\label{P:conjugation}
\begin{enumerate}
\item
Let $(L,\Fil)$ be a $\cO_K$-lattice in $(D,\Fil)$ with induced filtration,
let $f$ be an automorphism of $(D,\Fil)$ and let $L'=f(L)$.
Write $L=\pscal{e_1,\ldots,e_d}$ and $L'=\pscal{e'_1,\ldots,e'_d}$
for bases of $L$ and $L'$
adapted to the filtration, respectively.
Then
$$
f(e_1,\ldots,e_d) = (e'_1,\ldots,e'_d)\cdot C
$$
for some $C\in\cP_\k$.
\item
If $D=\pscal{e_1,\ldots,e_d}$ is a basis of $D$ adapted to the filtration,
then there is $C\in\cP_\k$ such that
$(e'_1,\ldots,e'_d):= (e_1,\ldots,e_d)\cdot C$ forms a basis for $L$
that is adapted to the filtration.
\item
Let $L$ be a $\varphi$-lattice of $(D,\Fil,\varphi)$.
Let $\Phi$ and $\Phi'$
be matrices of $\varphi$ with respect to bases $\pscal{e_1,\ldots,e_d}$
and $\pscal{e_1',\ldots,e_d'}$
of $L$ both adapted to the filtration, respectively. Then
$\Phi' = C^{-1} \Phi C^\sigma$ for some $C\in\cP_\k$.
\end{enumerate}
\end{proposition}
\begin{proof}
By Lemma \ref{L:parabolic}(1), $\Mat(f)\in\cP_{\k,K}$.
But $f$ is an isomorphism between two $\cO_K$-lattices, so
$\Mat(f) \in \GL_d(\cO_K)$
and hence $\Mat(f)\in\cP_\k$.
Then part (2) follows immediately by
setting $f:e_j\mapsto e_j'$.
Let $\theta$ be the automorphism of $L$ that
$\theta(e_1,\ldots,e_d)=(e'_1,\ldots,e'_d)C$ as in part (1).
Since $\theta\varphi = \varphi\theta$, and  $\varphi$ is
$K$-semilinear, we have $C \Phi' = \Phi C^\sigma$ where $\sigma$ is the
absolute Frobenius on $K$.
This proves part (3).
\end{proof}

\subsection{Construction of crystalline representations}
\label{S:2.2}

In this subsection we let $(D,\Fil)$ be given with
$\HP(D/K)=(k_1,\ldots,k_d)=:\k$, where $k_1\geq \ldots \geq k_d$, of
dimension $d$ as $K$-vector space. We write
$\Delta_\k:= \diag(p^{k_1},\ldots,p^{k_d})$. For any matrix
$C=(c_{ij})_{1\leq i,j\leq d}\in \GL_d(\cO_K)$, define
$C^\flat:=\Delta_\k C \Delta_\k^{-1} = (p^{k_i-k_j}c_{ij})_{1\leq
i,j\leq d}$. If $C\in \cP_\k$ then clearly $C^\flat\in
\cP_\k$. We define two matrices $A,B$ in $\GL_d(\cO_K)$
being {\em parabolic equivalent} with respect to $\k$,
denoted by $A\sim_\k B$, if $A=C^{-1}BC^\flat$ for some matrix $C$ in
$\cP_\k$; We say they are {\em $\sigma$-linear parabolic
equivalent with respect to $\k$}, denoted by $A\sim_{\sigma,\k} B$, if $A=C^{-1} B
(C^\flat)^\sigma$. Note that if $\HP(D/K)=(0,\ldots,0)$ then
$A\sim_{\sigma,\k} B$ if and only if $A=C^{-1}BC^{\sigma}$.
We denote the orbits of the equivalence classes by
$\GL_d(\cO_K)/\sim_\k$.
Namely, it can be considered the orbits of
map $A\rightarrow C^{-1}A(C^\flat)^\sigma$ in $\GL_d(\cO_K)$ with $C\in \cP_\k$.

\begin{lemma}\label{L:isogeny}
For any $K$-module $D'$ of dimension $d$, let
$f:(D,\Fil)\rightarrow D'$ be a $K$-semilinear (or $K$-linear)
isomorphism of $K$-modules. Let $M,M'$ be $\cO_K$-lattices in
$D,D'$ respectively with a map $f|_M: M\rightarrow M'$.
If $f(\Fil^i M )\subseteq p^i M'$ and $\sum_{i\in\Z}p^{-i}f(\Fil^i M)
= M'$, then there exists a basis for $M$ adapted to the filtration
such that $\Mat(f|_M)= A\Delta_\k$ with $A\in \GL_d(\cO_K)$.
The converse also holds.
\end{lemma}
\begin{proof}
Choose a basis of $M=\pscal{e_1,\ldots,e_d}$ adapted to its
filtration (see Proposition \ref{P:basis}). Let
$M'=\pscal{e_1',\ldots,e_d'}$ be any basis of $M'$. Let $r_i=\dim
\Fil^{k_i}M$. Note that for every $1\leq i\leq d$ we have
\begin{eqnarray}\label{E:Mat(f)}
f(e_1,\ldots,e_{r_i})  &=&
(e'_1,\ldots,e'_d) \cdot\Mat(f)\cdot
\left(
\begin{array}{c}
  \Id_{r_i\times r_i} \\
  0 \\
\end{array}
\right)_{d\times r_i}.
\end{eqnarray}
By induction on $i$ the first condition
in this lemma is equivalent to that $A:=\Mat(f)\cdot\Delta_{\k}^{-1}
\in M_{d\times d}(\cO_K).$ The second condition says that the column vectors
in the matrix $A$ generate $M'$ over $\cO_K$, that is $A\in
\GL_d(\cO_K)$. Changing basis for $M'$ clearly does not change the
form of $\Mat(f)=A\Delta_{\k}$. Suppose we change the basis of $M$ from
$\pscal{e_j}$ to $\pscal{b_j}$, then $(b_1,\ldots,b_d) =
(e_1,\ldots,e_d)C$ for some $C\in \cP_\k$ by Proposition
\ref{P:basis}(2): Suppose $\varphi$ is $K$-semilinear (the case
that $\varphi$ is $K$-linear is similar), then by
Proposition \ref{P:conjugation}(3) the matrix of $f$ with respect
to basis $\pscal{b_j}$ is $\Mat(f|_{\pscal{b_j}}) = C^{-1}A\Delta_\k
C^\sigma = C^{-1}A (C^\flat)^\sigma \Delta_{\k}$. It is clear that
$C^{-1}A(C^\flat)^\sigma\in \GL_d(\cO_K)$.

Conversely, suppose $\Mat(f|_M)= A\Delta_\k$. It suffices to show that
$f(\Fil^{k_i}M) \subseteq p^{k_i} M'$ for all $i$
and $\sum_{i}p^{-k_i} f(\Fil^{k_i} M) = M'$.
By (\ref{E:Mat(f)}) we have that
\begin{eqnarray*}
f(e_1,\ldots,e_{r_i}) &=& (e'_1,\ldots,e'_d)A \Delta_\k
\left(
\begin{array}{c}
  \Id_{r_i\times r_i} \\
  0 \\
\end{array}
\right)_{d\times r_i}\\
&=&
(e'_1,\ldots,e'_d)A
\left(
\begin{array}{c}
  \diag(p^{k_1},\ldots,p^{k_i})
  \\
  0 \\
\end{array}
\right)_{d\times r_i}.
\end{eqnarray*}
Because $k_1\geq \ldots\geq k_d$ we have
that $f(\Fil^{k_i}M)\subseteq p^{k_i} M'$.
This also shows that
for every $k_i$ we have
\begin{eqnarray}\label{E:f-2}
p^{-k_i} f(e_1,\ldots,e_{r_i})
&=& (e'_1,\ldots,e'_d) A
\left(
\begin{array}{c}
  \diag(p^{k_1-k_i},\ldots,1)
  \\
  0 \\
\end{array}
\right)_{d\times r_i}.
\end{eqnarray}
Since $A\in\GL_d(\cO_K)$
The image in (\ref{E:f-2}) generates
the submodule $\pscal{e'_{r_{i+1}+1},\ldots,e'_{r_i}}$ of $M'$.
Hence the sum of these submodules generates the
entire $M'$.
\end{proof}

Following Fontaine-Laffaille's notations, a
{\em strongly divisible $\varphi$-lattice} $L$ over $\cO_K$ in $D$ is
a free $\cO_K$-submodule in $D$ such that
$$
L\otimes_{\cO_K}K = D,\quad
\varphi(L)\subseteq L, \quad \sum_{i\in\Z} p^{-i}\varphi(\Fil^i D\cap L) = L.
$$

\begin{corollary}\label{C:SDL}
Let $L$ be a $\varphi$-lattice in $(D,\Fil,\varphi)$ with
induced filtration. Then $(L,\Fil,\varphi)$ is strongly divisible
if and only if there is a basis of $L$ adapted to the
filtration such that $\Mat(\varphi|_L) = A\Delta_\k$ for some
$A\in\GL_d(\cO_K)$.
\end{corollary}
\begin{proof}
This follows by applying Lemma \ref{L:isogeny} to the
$K$-semilinear automorphism $\varphi$ of $D$.
\end{proof}

For any matrix $M$ in $\GL_d(\cO)$ where $\cO$ is a principal ideal domain
(e.g., $\cO=K$ or $\bkplus$), there exists invertible matrices $C,D$ in
$\GL_d(\cO)$ such that $CMD$ is a diagonal matrix in $M_{d\times d}(\cO)$ with
entries on the diagonal equal to $p$-powers with non-increasing $p$-adic valuations.
We call the (ordered) $d$-tuple of diagonal entries of $CMD$
the {\em Hodge polygon of matrix} $M$.
If $f$ is an endomorphism of free modules $L$ over $\cO$ then we denote
$\HP(f|_L):=\HP(\Mat(f))$. This Hodge polygon is related
to Hodge polygon of filtered modules we defined earlier.

\begin{proposition}\label{P:parabolic}
For any $\k$, there is a bijection $\Xi: \GL_d(\cO_K)/\sim_{\sigma,\k}
\longrightarrow \MF{ad,\k}(K)$ by sending every $A\in\GL_d(\cO_K)$ to
the filtered $\varphi$-module with $\Mat(\varphi|_D)=A\Delta_\k$ with respect
to a basis adapted to filtration.
\end{proposition}

\begin{proof}
(i) Suppose $(D,\Fil,\varphi)$ is weakly admissible, then there is
a strongly divisible $\varphi$-lattice $L$ such that
$\Mat(\varphi|_L)=A\cdot \Delta_\k$. Any isomorphism
$(D,\Fil,\varphi)\longrightarrow (D,\Fil,\varphi')$ gives rise to
a (strict) isomorphism $f: L\rightarrow L'$ as filtered
$\varphi$-modules that is $\varphi-$ and $G(K/\Qp)$-equivariant
where $L'=f(L)$. Since $f$ commutes with $\varphi$ and $\Fil^i$ it
follows from the definition above Corollary \ref{C:SDL}
that $L'$ is necessarily a strongly
divisible $\varphi$-lattice in $(D,\Fil,\varphi')$. Hence by
Proposition \ref{P:basis} we may choose bases of $L$ and $L'$ such
that $\Mat(\varphi|_L)= A\cdot\Delta_\k$ and $\Mat(\varphi|_{L'}) =
A'\Delta_\k$ where $A,A'\in\GL_d(\cO_K)$. Since $f(\Fil^iL ) =
\Fil^iL'$ for all $i$ we have $\Mat(f|_L)=Q\in\cP_\k$ with
respect to any bases of $L$ and $L'$ adapted to the filtration on
$D$. This $Q$ is unique up to conjugation in $\cP_\k$ by
Proposition \ref{P:conjugation}. But since $f \varphi = \varphi'
f$ on $L$, and $\varphi$ is $K$-semilinear, we have $Q A\Delta_\k =
A'\Delta_\k Q^\sigma$. Hence we have $A' = C^{-1} A \Delta_\k C^\sigma
\Delta_{\k}^{-1} = C^{-1} A (C^\flat)^\sigma$. This proves one
direction of the correspondence.

(ii)
Let $L=\pscal{e_i}_i$ be a $\cO_K$-lattice in $(D,\Fil)$
with basis adapted to its induced filtration.
For any $A\in\GL_d(\cO_K)$
define an endomorphism of $L$ by $\Mat(\varphi|_L):=A\Delta_\k$.
Because $\det(\Mat(\varphi|_L))\neq 0$, the map
$\varphi$ induces an automorphism of $(D,\Fil)$.
By Corollary \ref{C:SDL}, $L$ is a strongly divisible
$\varphi$-lattice in $(D,\Fil,\varphi)$ with induced filtration.
We let $(D,\Fil,\varphi')$ be associated to another matrix
$A'\in\GL_d(\cO_K)$ in the same way, and it is weakly admissible as well.
Suppose $A\sim_{\sigma,\k} A'$ and $A'=C^{-1}A(C^\flat)^\sigma$ for some $C\in\cP_\k$.
Then the matrix $C^{-1}\in\cP_\k$ defines
an automorphism $\theta$ of $(D,\Fil)$ by Lemma \ref{L:parabolic}.
The hypothesis $A'=C^{-1}A (C^\flat)^\sigma$ is equivalent to
that $\theta \varphi = \varphi' \theta$.
Thus $\theta$ gives rise to an
isomorphism as filtered $\varphi$-modules.
This proves the other direction of the above correspondence.
\end{proof}

\begin{proof}[Proof of Theorem A]
Let $\vcris^*$ be the quasi-inverse of functor $\dcris^*$.
Composing the bijection $\Xi$ in Proposition \ref{P:parabolic}
and the equivalence functor $\vcris^*$, we get the desired bijection
$$
\Theta: \GL_d(\cO_K)/\sim_{\sigma,\k}
\stackrel{\Xi}{\longrightarrow} \MF{ad,\k}(K)
\stackrel{\vcris^*}{\longrightarrow} \Rep_{\cris/\Qp}^\k (G_K).$$
\end{proof}

\begin{remark}
[Metric on $\Rep_{\cris/\Qp}^\k(G_K)$]\label{R:metric}
The bijection $\Theta$ of Theorem A allows us to define a non-Archimedian metric
($p$-adic) on $\Rep_{\cris/\Qp}(G_K)$  by endow such metric on
$\GL_d(\cO_K)/\sim_{\sigma,\k}$.
Namely, let $|A-A'|_p:=p^{-\ord_p(A-A')}$ be the
$p$-adic metric on $\GL_d(\cO_K)$ for any $A,A'\in\GL_d(\cO_K)$.
Let $|[A]-[A']|_p$ be a quotient metric on $\GL_d(\cO_K)/\sim_{\sigma,\k}$
defined by
$$|[A]-[A']|_p =
\sup_{C_1,C_2\in\cP_\k} |C_1^{-1}AC_1^{\flat\sigma}-C_2^{-1}A'C_2^{\flat\sigma}|_p.$$
For any $V,V'\in \Rep_{\cris/\Qp}^\k(G_K)$, define $\dist(V,V') := |[A]-[A']|_p$.
It is clear that $\GL_d(\cO_K)$ is compact as $\cO_K$ is compact under $|\cdot|_p$.
With quotient metric, $\GL_d(\cO_K)/\sim_{\sigma,\k}$ is also compact.
Hence our metric defined above makes $\Rep_{\cris/\Qp}^\k(G_K)$
a compact space.
\end{remark}

\subsection{Irreducibility in $\Rep_{\cris/\Qp}(G_K)$}
\label{S:2.3}

Let $\k=(k_1,\ldots,k_d)$
and  $\k'=(k'_1,\ldots,k'_{d'})$ be Hodge polygons.
If $\k'$ is a sub-polygon of $\k$,
then we denote this partial order by $\k'\subseteq \k$. For example
$(3,2,2,0)\subseteq (5,4,3,2,2,1,0)$.
Let $\Theta: \GL_d(\cO_K)
\longrightarrow \Rep_{\cris/\Qp}^\k(G_K)$
be the bijection in Theorem A.

\begin{proposition}\label{P:irreducibility}
Let $V\in \Rep_{\cris/\Qp}^\k (G_K)$ be such that $\Theta(A)=V$.
Then $V$ is irreducible if and only if
for any $\k'\subsetneq \k$ there is no $A'\in\GL_{d'}(\cO_K)$ and
no $C\in M_{d\times d'}(\cO_K)$ of rank $d'$ such that
$C=A\Delta_\k C^\sigma (A'\Delta_{\k'})^{-1}$.
\end{proposition}
\begin{proof}
That $V'$ is a proper subobject of $V$ if
and only if $D'=\dcris^*(V')$ is a proper subobject of $D=\dcris^*(V)$.
By Corollary \ref{C:SDL},
there is a strongly divisible $\varphi$-lattice $L$ in $D$ such that
$\Mat(\varphi|_L)=A\Delta_\k$ for $\Theta(A)=V$ with respect to a basis adapted to
filtration $L=\pscal{e_1,\ldots,e_d}$.
That $V'\subseteq V$ implies that
their corresponding $\HP$ has $\k'\subsetneq \k$.
We can choose a strongly divisible lattice $L'$ in $V'$ such that
$(e'_1,\ldots,e'_{d'}) =
(e_1,\ldots,e_d)C$ is a basis in $L'$ adapted to filtration for some
$C\in M_{d\times d'}(\cO_K)$ of rank $d'$, and such that
$\Mat(\varphi|_{L'})=A'\Delta_{\k'}$ with $A'\in\GL_{d'}(\cO_K)$.
Thus by simple semi-linear algebra we have
$CA'\Delta_{\k'}=A\Delta_\k C^\sigma$.
Thus our assertion follows.
\end{proof}

A generalization of Proposition \ref{P:irreducibility}
to $\Rep_{\cris/E}(G_K)$ is given in Proposition \ref{P:irreducibility-2}.

\begin{remark}\label{R:irreducible}
Let $V$ be any $2$-dim crystalline representation
of $G_K$ whose Hodge polygon is $(k,0)$ with $k>0$.
Then it is known that $V$ is irreducible if and only if
$\ord_p\tr(\Mat(\varphi))=\ord_p \tr(A\cdot\diag(p^k,1))>0$
Our criterion in Proposition \ref{P:irreducibility}
can be used to recover Breuil's classification
(see \cite{Br02}), namely, $V$ is irreducible if and only if there is
an unramified character $\eta$ such that
$\Theta^{-1}(V\otimes \eta)=
\left(
  \begin{array}{cc}
    0 & -1 \\
    1 & a_p \\
  \end{array}
\right)
$
for some $a_p\in\m_K$.
Meanwhile $V$ is reducible if and only if there is an
unramified character $\eta$ such that
$\Theta^{-1}(V\otimes\eta) =
\left(
  \begin{array}{cc}
    1 & 0 \\
    v & u \\
  \end{array}
\right)
$
for some $u\in\cO_K^*$ and $v\in\cO_K$.
A generalization of this classification to $\Rep_{\cris/E}(G_K)$
in 2-dimensional case lies in Theorem \ref{T:irreducible}(or Theorem D), see more discussion in Section \ref{S:5.1}.
\end{remark}

\section{Embedded strongly divisible $\varphi$-lattices}
\label{S:3}

We assume $K=\Q_{p^a}$ and $E$ is an extension of $K$.
A $\varphi$-action on $K\otimes_\Qp E$-module $D$
is an automorphism on $D$ that is
defined by $\varphi((c_K\otimes c_E)y) = (\sigma(c_K)\otimes c_E)\varphi(y)$ for
$c_K\in K, c_E\in E$ and $y\in D$.
Let $\MF{}(K\otimes_\Qp E)$ denote the category of filtered
$\varphi$-modules over $K\otimes_\Qp E$. This category is extensively
studied in the work by Breuil and Mezard \cite{BM02}.
See also \cite{Sa04} for potential version of Breuil-Mezard.
There is a natural
forgetful functor $\MF{}(K\otimes_\Qp E)\longrightarrow \MF{}(K)$ by
forgetting $E$-module structure. A free filtered $\varphi$-module
$D$  in $\MF{}(K\otimes_\Qp E)$ is {\em weakly admissible} if it is
weakly admissible as a $K$-module in $\MF{}(K)$ by forgetting the
$E$-vector space structure. Let $\MF{ad}(\cdot)$ denote the
subcategory of $\MF{}(\cdot)$ consisting of weakly admissible
filtered $\varphi$-modules.
At present an effective and
convenient criterion for weakly admissibility (when $K$ is
unramified) is via the existence of strongly divisible
$\varphi$-lattices in $D$ (see \cite{Laf80} and \cite{FL82}).
This section explores a generalization of this theory from $\MF{ad}(K)$ to
$\MF{ad}(K\otimes_\Qp E)$.

We define two sets of Newton and Hodge polygons for each
$D\in\MF{}(K\otimes_\Qp E)$, embedded Newton and Hodge polygon
in Section \ref{S:3.1} and $\sigma$-invariant
Newton and Hodge polygons in Section 3.2.
Our $\sigma$-invariant polygons are essentially the polygons of
$D$ considered \cite[Section 3]{BS07} after normalized by a factor $1/a$.

\subsection{Embedded Newton and Hodge polygons}
\label{S:3.1}

For Sections \ref{S:3} and \ref{S:4},
let $\k=(\k_j)_{j\in\Z/a\Z}$ where
$\k_j=(k_{j1},\ldots,k_{jd})$ with $k_{j1}\geq\cdots\geq k_{jd}$.
Write $\Delta_{\k_j}:=\diag(p^{k_{j,1}},\ldots,p^{k_{j,d}})$.
To be self-contained, we recall some preliminaries in this section, see
\cite{BM02} for reference. Let $D\in \MF{}(K\otimes_\Qp E)$.
We identify the following group
$\Gal(\Q_{p^a}/\Qp)=\pscal{\tau^i}_{i\in\Z/a\Z}
\simeq\Z/a\Z$ that sends each $\tau^i$ to $i$.
There is a natural isomorphism $\theta: K\otimes_{\Qp} E \longrightarrow
\prod_{j\in\Z/a\Z}E$ defined by $x\otimes y\mapsto
(\sigma^j(x)y)_{j\in\Z/a\Z}$, where $\sigma$ is the absolute Frobenius on $K$.
Let $e_{\sigma^j}:=(0,\cdots,1,\cdots,0)$ with $1$ at $j$-th position
in $\prod_{j\in\Z/a\Z} E$ and let $D_j:= e_{\sigma^j} D$.
Then we have an isomorphism $\theta: D\longrightarrow\prod_j D_j$
as $E$-vector spaces.
For any  $j\in\Z/a\Z$, let $(D_j,\Fil)$ be induced from $(D,\Fil)$,
that is, $\Fil^i D_j: = \Fil^i D\cap D_j = e_{\sigma^j}\Fil^iD$.
It is a filtered module of dimension $d$ over $E$. Write
$\HP(D_j/E)=\k_j$. We call $\HP_{\Z/a\Z}(D):= \HP(D_j/E)_{j\in\Z/a\Z}$ the {\em embedded Hodge polygon} of $(D,\Fil)$. We say $D$ is {\em effective} if $k_{jd}\geq 0$ for all $j$. In the literature the set $\{k_{j1},\ldots,k_{jd}\}$ (without order) is also called {\em Hodge-Tate weight with respect to} $\tau^j$
(see for instance the notation of Gee \cite[Section 4.2]{Ge08b}).

By assuming $\theta\varphi=\varphi\theta$ on $D$ we have for every
$x\otimes y \in K\otimes_\Qp E$ and $v\in D$ that
\begin{eqnarray}\label{E:semilinear}
\varphi ((\sigma^{0}x)y v,\ldots,(\sigma^{a-1}x)yv)
&=&
((\sigma^1x)y\varphi(v),\ldots,(\sigma^{0}x)y\varphi(v)).
\end{eqnarray}
Extending by linearity this defines a map $\varphi$ on $\prod_j D_j$.
This yields an isomorphism $\theta: D\longrightarrow\prod_j D_j$
in $\MF{}(K\otimes_\Qp E)$.

Below we shall consider the lattices in $D \in \MF{}(K\otimes_\Qp E)$.
Recall Fontaine's notation that a morphism $f: (L,\Fil)\rightarrow (L',\Fil)$ between filtered modules is {\em strict} if $f(\Fil^i L ) = f(D)\cap \Fil^i L'$, where $L$ is an $\cO_K\otimes_\Zp \cO_E$-module or $K\otimes_\Qp E$-modules.
A $\varphi$-{\em lattice} $L$ over $\cO_K\otimes_\Zp \cO_E$ in $D$ with induced
filtration is a  $\cO_K\otimes_\Zp \cO_E$-module in $D$ such that
$$
L\otimes_{\cO_K\otimes_{\Zp} \cO_E} (K\otimes_\Qp E)  = D,\quad
\varphi(L)\subseteq L, \quad \Fil^i L  = \Fil^i D\cap L.
$$
It is a {\em strongly divisible $\varphi$-lattice} if, in addition,
it satisfies $\varphi(\Fil^i L)\subseteq p^i L$ and
$$
\sum_{i\in \Z} p^{-i}\varphi(\Fil^i L) = L.
$$
For $K=\Qp$ our strongly divisible lattice is exactly the one
defined in Fontaine-Laffaille's theorem.

\begin{lemma}\label{L:split}
The induced map $\theta: \cO_K \otimes_{\Zp}
\cO_E\longrightarrow \prod_{j\in\Z/a\Z}\cO_E$ is a
$\cO_K\otimes_{\Zp}\cO_E$-algebra isomorphism compatible with
$\varphi$ and $G(K/\Qp)$-actions.
Moreover, we have $L\cong \prod_{j\in\Z/a\Z} L_j$ with
$L_j=L\cdot e_{\sigma_j}$ over $\cO_K\otimes_\Zp\cO_E$.
\end{lemma}
\begin{proof}
It suffices to show the first statement.
As a complete discrete valuation ring, $\cO_K$ is generated by
one element over $\Zp$ and $\cO_K\cong\Zp[x]/f(x)$ for some irreducible
$f(x)\in\Zp[x]$ of degree $a$. Then
$$\cO_K\otimes_\Zp \cO_E\cong \Zp[x]/f(x) \otimes_\Zp \cO_E
\cong \cO_E[x]/f(x) \cong \prod_{j\in\Z/a\Z}\cO_E$$
since roots of irreducible polynomial $f(x)\in\Zp[x]$
lies in $\cO_K$ and hence $\cO_E$, since $K\subseteq E$.
\end{proof}

Obviously the endomorphism $\varphi$ of $D$
does not induce an endomorphism on $D_j$.
However, from (\ref{E:semilinear})
we have that $\varphi(D_j)\cong D_{j-1}$.
Thus we have $\varphi(L_j) \subseteq  \varphi(L)\cap \varphi(D_j)
\subseteq L \cap D_{j-1}=L_{j-1}$. Since $[L:\varphi(L)]$ is
finite, the inclusion
$\varphi(L_j) \subseteq L_{j-1}$ is of finite index for every $j\in\Z/a\Z$.

\begin{proposition}\label{P:SDL}
Let $D$ be a free filtered $\varphi$-module in $\MF{}(K\otimes_\Qp
E)$. Any $\varphi$-lattice $L$
over $\cO_K\otimes \cO_E$ in $D$ is strongly divisible if and
only if on the embedded form $L\cong \prod_{j\in\Z/a\Z}L_j$ we
have $\varphi(\Fil^iL_j)\subseteq p^iL_{j-1}$ and
\begin{eqnarray}\label{E:SDL}
\sum_{i\in \Z} p^{-i} \varphi(\Fil^i L_j) &=& L_{j-1}
\end{eqnarray}
for all $j\in\Z/a\Z$.

Suppose $D$ is weakly admissible then a
$\varphi$-lattice $L$ of $D$ is strongly divisible if for every
$j\in\Z/a\Z$ that
$\sum_{i\in\Z} p^{-i} \varphi(\Fil^i
L_j)\subseteq L_{j-1}$.
\end{proposition}

\begin{proof}
By Lemma \ref{L:split},
we have $\theta: \Fil^i L \cong \prod_{j\in\Z/a\Z} \Fil^i L_j$
and $\varphi \prod_j  L_j  = \prod_j \varphi(L_j)$. Hence,
\begin{eqnarray*}
p^{-i} \varphi(\Fil^i L) & = & p^{-i}
\varphi(\Fil^i \prod_j L_j)= \prod_j  p^{-i} \varphi(\Fil^i L_j).
\end{eqnarray*}
So,
\begin{eqnarray}\label{E:SDL-1}
\sum_{i\in\Z}p^{-i}\varphi(\Fil^i L)
&=& \sum_{i\in\Z}\prod_j p^{-i} \varphi(\Fil^i L_j).
\end{eqnarray}
By definition $L$ is strongly divisible if and only if
\begin{eqnarray}\label{E:SDL-2}
\sum_{i\in\Z}p^{-i}\varphi(\Fil^i L)
&=&\prod_j L_{j-1}.
\end{eqnarray}
Comparing the right-hand-sides of the equations (\ref{E:SDL-1}) and (\ref{E:SDL-2}),
we have that $L$ is strongly divisible if and only if
$\sum_{i\in\Z}p^{-i}\varphi(\Fil^i L_j) = L_{j-1}$ for
every $j\in\Z/a\Z$.

Suppose $D$ is  weakly admissible, then by the hypothesis we have
$$\sum_{i\in\Z} p^{-i} \varphi(\Fil^i L) \subseteq L$$ and hence $L$
is strongly divisible by Fontaine-Laffaille theory.
\end{proof}

Let $L$ be a $\varphi$-lattice in $D$ with induced filtration. Let
$\cP_{\k_j}$ be the integral parabolic group with respect to $\k_j$, as defined in
Section 2.1. We say two $a$-tuples $(A_0,\ldots,A_{a-1})$ and $(A_0',\ldots,A'_{a-1})$ with matrices $A_j,A'_j \in \GL_d(\cO_E)$ are {\em parabolic equivalent with respect to $\k=(\k_j)_{j\in\Z/a\Z}$}
if there exists an $a$-tuple $(C_0,\ldots,C_{a-1})$
with $C_j\in \cP_{\k_j}$ such that $A'_j =
C_{j-1}^{-1}A_jC_j^\flat$ for every $j\in\Z/a\Z$.
We denote it by
$(A_j)_{j\in\Z/a\Z}\sim_\k (A'_j)_{j\in\Z/a\Z}$.

We remark that this equivalence does not imply that
the equivalence $A_j\sim_{\k_j} A_j'$
(as defined in \S \ref{S:2.2}) for every $j$ when $a\geq 2$.

\begin{proposition}\label{P:simultaneous}
Let $D\in\MF{}(K\otimes_\Qp E)$ be free with filtration data
$\HP_{\Z/a\Z}=(\k_j)_{j\in\Z/a\Z}$. Write $\Delta_{\k_j}:=
\diag(p^{k_{j1}},\ldots,p^{k_{jd}})$ for every $j\in\Z/a\Z$.
Then the following are equivalent
\begin{enumerate}
\item $D$ is weakly admissible;
\item there exists a strongly divisible $\varphi$-lattice $L$ over $\cO_K\otimes \cO_E$
in $D$;
\item there exists a $\varphi$-lattice $L=\prod_j L_j$ over $\cO_K\otimes \cO_E$ in $D$
such that with respect to a suitable basis we have
$\Phi_j:=\Mat(\varphi|_{L_j}) =  A_j \Delta_{\k_j}$ for some $A_j\in\GL_d(\cO_E)$.
\end{enumerate}
\end{proposition}
\begin{proof}
Laffaille's theorem \cite{Laf80}
asserts that (2)$\Longrightarrow$(1). By
Laffaille's theorem again, (1) implies there is a strongly
divisible $\varphi$-lattice $L'$ over $\cO_K$. Let
$L=L'\cO_E := \sum_{i=1}^{d}L'e_i$ where $\cO_E=\pscal{e_1,\ldots,e_d}$ over $\Zp$.
It is clearly a module over $\cO_K\otimes_\Zp\cO_E$ in the natural way.
Since filtration is $E$-linear,
we can show that $L$ is strongly divisible
over $\cO_K\otimes_\Zp \cO_E$ as desired in (2), indeed
we have
$$\sum_{i\in\Z} p^{-i}\varphi(\Fil^i L ) =\sum_{i\in\Z} p^{-i}\varphi(\Fil^i L'\cO_E)
=\sum_{i\in\Z} p^{-i}\varphi(\Fil^i L')\cO_E = L'\cO_E.
$$

By Proposition \ref{P:SDL} and Lemma \ref{L:isogeny},
it is easy to see that (2)$\Longleftrightarrow$(3).
\end{proof}

Let $\MF{ad,\k}(K\otimes_\Qp E)$ be the subcategory
of $\MF{ad}(K\otimes_\Qp E)$ consisting of
objects with $\HP_{\Z/a\Z}(D)=\k$.
Then we have the following result.

\begin{proposition}\label{P:parabolic-2}
There is a bijection between $\GL_d(\cO_E)^a/\sim_{\k}$
and $\MF{ad,\k}(K\otimes_\Qp E)$, which
sends $(A_0,\ldots,A_{a-1})$ in $\GL_d(\cO_E)^a$
to a strongly divisible $\varphi$-lattice $L=(L_j)_j$
such that $\Mat(\varphi|_{L_j})= A_j\Delta_{\k_j}$ with
respect to a basis adapted to the filtration.
\end{proposition}
\begin{proof}
Consider an isomorphism $f: (D,\Fil,\varphi)\rightarrow
(D,\Fil,\varphi')$ in the category $\MF{}(K\otimes_\Qp E)$. Let
$L, L'$ be strongly divisible $\varphi$-lattice
($\varphi'$-lattice) in $D,D'$ respectively where $L'=f(L)$.
By Proposition \ref{P:simultaneous}, we
have $\Mat(\varphi|_{L_j}) = A_j \Delta_j$ and
$\Mat(\varphi|_{L'_j}) = A'_j \Delta_{k_j}$ for some $A_j\in
\GL_d(\cO_E)$. The restriction maps $f(\Fil^i D_j ) =\Fil^i D_j$
and $f(L_j)=L'_j$ assure that $\Mat(f|_{L_j}) = :Q_j \in
\cP_{\k_j}$ by Proposition \ref{P:conjugation}(1). Since
$f \varphi =  \varphi'f$ on $L_j$, we have $ Q_{j-1} (A_j\Delta_{\k_j})
= (A'_j \Delta_{\k_j}) Q_j$. That is, $A'_j = Q_{j-1} A_j (\Delta_{\k_j}
Q_j^{-1} \Delta_{\k_j}^{-1})$, i.e., $A'_j = C_{j-1}^{-1} A_j
C_j^\flat$.
We leave the converse direction as an exercise to the readers.
\end{proof}

It is easy to see that isomorphism classes in $\MF{ad}(K\otimes_\Qp E)$
upon a twist of an unramified character is in 1-1 correspondence
with the set of parabolic equivalence classes of $a$-tuple
$(A_0,\ldots,A_{a-1})$ upon multiplication by units $(u_0,\ldots,u_{a-1})$
componentwise.

Let $\Rep_{\cris/E}^{\k}(G_K)$ denote the subcategory of $\Rep_{\cris/E}(G_K)$
consisting of objects with embedded Hodge polygon $\HP_{\Z/a\Z}=\k$.
By Proposition \ref{P:parabolic-2} we may define
a map
$$\Theta:
\GL_d(\cO_E)^a/\sim_\k  \longrightarrow    \Rep_{\cris/E}^\k (G_K)
$$
defined by composing
the equivalence functor $\vcris^*$ and the bijection given in
Proposition \ref{P:parabolic-2}. This is a generalization of Theorem A
to $\Rep_{\cris/E}(G_K)$, and we formulate this in the following theorem.

\begin{theorem}[Theorem A']\label{T:A-analog}
Let $\k=(\k_j)_{j\in\Z/a\Z}$ be embedded Hodge polygons.
Then there is a bijection
$$
\Theta: \GL_d(\cO_E)^a/\sim_\k \longrightarrow
\Rep^\k _{\cris/E}(G_K)
$$
which sends every $(A_j)_j\in\GL_d(\cO_E)^a$
to a crystalline $E$-linear representation $V$
whose $\dcris^*(V)$ contains embedded strongly divisible $\varphi$-lattice
$L=\prod_{j\in\Z/a\Z} L_j$ such that
$\Mat(\varphi|_{L_j}) = A_j \Delta_{\k_j}$, with respect to
a basis adapted to filtration.
\end{theorem}

\begin{remark}[Metric on $\Rep_{\cris/E}^\k(G_K)$]\label{R:metric-2}
Similar to Remark \ref{R:metric}, we define a metric to make $\Rep_{\cris/E}^\k(G_K)$ a compact metric space. Write $A=(A_j)_j$ and $A'=(A'_j)_j$ that represent $V$ and $V'$ in $\Rep_{\cris/E}(G_K)$ as given in Theorem A'. Let
\begin{eqnarray}\label{E:metric-2}
\dist(V,V') &:=& |[A]-[A']|_p = \sup_{B,B'}\max_{j\in\Z/a\Z}|B_j-B_j'|_p
\end{eqnarray}
where the sup ranges over all $B=(B_j)\sim_\k A$ and $B'=(B'_j)\sim_\k A'$.
\end{remark}

Let $V$ and $V'$ be in $\Rep_{\cris/E}(G_K)$ of dimension $d$ and $d'$,
of Hodge polygon $\HP_{\Z/a\Z}(V)=\k=(\k_j)_j=(k_{j1},\ldots,k_{jd})_j$
and $\HP_{\Z/a\Z}(V')=\k'=(\k'_j)_j=(k'_{j1},\ldots,k'_{j,d'})_j$, respectively.
Let $A:=(A_j)_j$ in $\GL_d(\cO_E)^a$
and $A':=(A_j')_j$ in $\GL_{d'}(\cO_E)^a$.
Let $\Theta(A)=V$ and $\Xi(A')=V'$.
Write
\begin{eqnarray*}
\k\oplus \k'&:=& (k_{j,1},\ldots,k_{j,d}, k'_{j,1},\ldots,k'_{j,d'})_{j\in\Z/a\Z},\quad\mbox{re-order if necessary},\\
\k\otimes\k'&:=& ((k_{j,i}+k'_{j,i'})_{i,i'})_{j\in\Z/a\Z}\quad \mbox{where $1\leq i\leq d, 1\leq i'\leq d'$,
                             re-order if necessary},\\
A\oplus  A' &:=& (\diag(A_j,A'_j))_{j\in\Z/a\Z}\in \GL_{d+d'}(\cO_E)^a,\\
A\otimes A' &:=& (A_j\otimes A'_j)_{j\in\Z/a\Z}\in \GL_{dd'}(\cO_E)^a.
\end{eqnarray*}

\begin{proposition}\label{P:building-blocks}
Let notations be as above.
Let $V=\Theta(A)$ be in $\Rep_{\cris/E}^\k(G_K)$
and $V'=\Theta(A')$ be in $\Rep_{\cris/E}^{\k'}(G_K)$.
Then $\Theta$ induces the following injective maps:\\
(1)
\begin{eqnarray*}
\Theta: (\GL_d(\cO_E)\oplus \GL_{d'}(\cO_E))^a/\sim_{\k\oplus\k'}
&\hookrightarrow &
\Rep_{\cris/E}^{\k\oplus\k'}(G_K)
\end{eqnarray*}
defined by $\Theta(A\oplus A')= V\oplus V'$
whose imagines are a subcategory of $\Rep_{\cris/E}^{\k\oplus\k'}(G_K)$
consisting of a direct sum from $\Rep_{\cris/E}^\k(G_K)$ and
$\Rep_{\cris/E}^{\k'}(G_K)$;\\
(2)
\begin{eqnarray*}
\Theta: (\GL_d(\cO_E)\otimes\GL_{d'}(\cO_E))^a/\sim_{\k\otimes\k'}
&\hookrightarrow &
\Rep_{\cris/E}^{\k\otimes \k'}(G_K)
\end{eqnarray*}
defined by $\Theta(A\otimes A')= V\otimes V'$,
whose imagines are a subcategory of $\Rep_{\cris/E}^{\k\otimes\k'}(G_K)$
consisting of a tensor product from $\Rep_{\cris/E}^\k(G_K)$
and $\Rep_{\cris/E}^{\k'}(G_K)$.
\end{proposition}

\begin{proof}
It suffices to show
\begin{itemize}
\item[(1)]
$\Theta(A\oplus A')=V\oplus V'$ and $\HP_{\Z/a\Z}(V\oplus V')=\k\oplus\k'$
\item[(2)]
$\Theta(A\otimes A') = V\otimes V'$ and $\HP_{\Z/a\Z}(V\otimes V') = \k\otimes \k'$.
\end{itemize}
We should give a complete proof for part (2) and omit the relatively easier and similar part (1).
Let $\dcris^*(V)=D=\prod_j D_j$ and $\dcris^*(V')=D'=\prod_jD'_j$.
Note that $\prod_j D_j \otimes \prod_j D'_j
=\prod_{j} D_j\otimes D'_j.$
The action of $\varphi$ with respect to a basis adapted to
filtration is
$$
\prod_{j}\Mat(\varphi|_{D_j}) \otimes \Mat(\varphi|_{D'_j})
=\prod_{j} (A_j\Delta_{\k_j}) \otimes (A'_j\Delta_{\k'_j})
=\prod_{j} (A_j\otimes A'_j)(\Delta_{\k_j} \otimes \Delta_{\k'_j}).
$$
One observes that the tensor product of two diagonal matrices
$$\Delta_{\k_j}\otimes \Delta_{\k'_j} =
\diag(p^{k_{j,1}+k'_{j,1}},\ldots, p^{k_{jd}+k'_{j,d'}}).$$
This shows that the bijection
$\GL_d(\cO_E)^a/\sim_\k
\longrightarrow \MF{ad,\k}(K\otimes_\Qp E)
$
in Proposition \ref{P:parabolic-2}
sends $A\otimes A'$ to $D\otimes D'$.
Apply the equivalence functor $\vcris^*$ on Tannekian categories, we have
$\Theta(A\otimes A')=\vcris^*(D\otimes D')=V\otimes V'$.
\end{proof}

For completeness and the convenience of the reader, we compile the following
facts. They are easy to verify by Theorem \ref{T:A-analog} and
Proposition \ref{P:building-blocks}.

\begin{remark}\label{R:dim-2}

\begin{enumerate}
\item 
Write $\vk=(k_0,\ldots,k_{a-1})\in \Z^a$.
Every unramified character $\eta$ lies in $\Rep_{\cris/E}^\vk (G_K)$
with $\vk=\0$. Namely, $\eta=\Theta(\vec{u})$ in $\Rep_{\cris/E}^\0 (G_K)$ for some $\vec{u}\in (\cO_E^*)^a$.

\item 
For any $\vk=(k_0,\ldots,k_{a-1})$ in $\Z_{\geq 0}^a$, write $\varepsilon_{a,\vk}:=\Theta((1,\ldots,1))$. In other words,
$\dcris^*(\varepsilon_{a,\vk})$ has
$\Mat(\varphi)=(p^{k_0},\ldots,p^{k_{a-1}})$
with respect to a basis adapted to filtration.

Every crystalline character in $\Rep_{\cris/E}^\vk (G_K)$ is of
the form $\varepsilon_{a,\vk}\otimes\eta$ for some unramified character $\eta$.

\item
For any embedded Hodge polygon $\k=(k_{j,1},\ldots,k_{j,d})_j$,
let $V\in\Rep_{\cris/E}^\k(G_K)$. Let $\vec{w}:=(-k_{0,d}, \ldots,-k_{a-1,d})$
then the embedded Hodge polygon $\HP_{\Z/a\Z}(V\otimes \varepsilon_{a,\vec{w}})=(k_{j,1}-k_{j,d},\ldots,0)_j$.
This verifies that we may always assume $k_{jd}=0$ for all $j$
by a twist.

\item
Let $V$ be a 2-dimensional crystalline representation with
embedded Hodge polygon $\HP_{\Z/a\Z}(V)=\k=(k_j,0)_j$.

\begin{enumerate}
\item If $V$ is decomposable (i.e., split), then
we have
$$V\cong
\left(
  \begin{array}{cc}
  \varepsilon_{a,\vk^{(1)}}    & 0 \\
   0  & \varepsilon_{a,\vk^{(2)}} \\
  \end{array}
\right) \otimes \eta
$$
for some $\vk^{(1)}\oplus \vk^{(2)} = \k$ and unramified character $\eta$.

\item
If $V$ is reducible then
$$
V\cong
\left(
  \begin{array}{cc}
    \varepsilon_{a,\vk^{(1)}} & \star \\
    0 & \varepsilon_{a,\vk^{(2)}} \\
  \end{array}
\right) \otimes \eta
$$
for some $\vk^{(1)}\oplus \vk^{(2)}=\k$ and unramified character $\eta$.
\item Irreducible $V$'s are discussed in Section \ref{S:5.1}.
\end{enumerate}
\end{enumerate}
\end{remark}

\subsection{Newton and Hodge polygons and their
$\sigma$-invariant variation}
\label{S:3.2}

We define $\sigma$-invariant Newton and Hodge polygons of
the embedded filtered $\varphi$-module $D=\prod_jD_j$  over
$K\otimes_\Qp E$ in this subsection.
They will play a vital role in our Fontaine-Laffaille type theorem
developed in this paper, see for example Theorem B. These definitions are classical and exist already in the literature under some disguise, see for example, \cite[Introduction]{MV07} and \cite[Section 3]{BS07}.

Let $\HP_{\Z/a\Z}(D)=\k=(k_{j,1},\ldots,k_{j,d})_{j\in\Z/a\Z}$.
For every $i$ let $\bar{k}_i:=\frac{1}{a}\sum_{j=0}^{a-1}k_{ji}$, then we call the
$d$-tuple $(\bar{k}_1,\ldots,\bar{k}_d)$ in $\Q^d$ the $\sigma$-{\em
invariant Hodge polygon} of $(D,\Fil)$ over $K\otimes_\Qp E$,
denoted by $\HP^\sigma(D/K\otimes_F E)$.
Note that $\MaxSlope(\HP^\sigma(D/K\otimes_\Qp E) = \bar{k}_1-\bar{k}_d$,
and it is equal to $\bar{k}_1$ when we assume all $k_{jd}=0$ (see Remark \ref{R:dim-2}(3)). Correspondingly, we define the $\sigma$-invariant Newton polygon by
$$\NP^\sigma(D/K\otimes_\Qp E):= \frac{1}{a}\NP(\varphi^a|D_0).$$
We define
\begin{eqnarray*}
t_H(D/K\otimes_\Qp E)&:=&\frac{1}{a}\sum_{j\in\Z/a\Z}\sum_{i\in\Z}i\cdot \dim_E(\Fil^i(D_j)/\Fil^{i+1}(D_j)),\\
t_N(D/K\otimes_\Qp E)&:=&\frac{1}{a}\ord_p({\det}_{\Q_{p^a}}(\varphi^a|_{D_0})).
\end{eqnarray*}
Then we have
\begin{eqnarray*}
t_H(D/K\otimes_\Qp E) &=& \sum_{i=1}^{d}\bar{k}_i=\sum_{j\in\Z/a\Z}\sum_{i=1}^{d}k_{ji},\\
t_N(D/K\otimes_\Qp E) &=& \sum_{i=1}^{d}\bar{\alpha}_i
\end{eqnarray*}
where
$\NP^\sigma(D/K\otimes_\Qp E)=(\bar{\alpha}_1,\ldots,\bar{\alpha}_d).$

Let $t_H(D)$ and $t_N(D)$ be defined as in \cite[Section 3]{BS07},
then one can see that $t_H(D/K\otimes_\Qp E) =\frac{1}{a}t_H(D)$ and $t_N(D/K\otimes_\Qp E) = \frac{1}{a}t_N(D)$.
In other words, the $\sigma$-invariant polygons are $1/a$ multiple of the polygon defined by Breuil-Schneider. We formulate these relations as follows:

\begin{proposition}\label{P:NP>HP}
The following statements are equivalent
\begin{enumerate}
\item[(i)]  $(D,\Fil,\varphi)$ is weakly admissible.
\item[(ii)]  $\NP^\sigma(D'/K\otimes_\Qp E) \prec \HP^\sigma (D'/K\otimes_\Qp E)$ for every sub-object $D'$ of $D$, and equality holds for $D$.
\item[(iii)]  the inequality $t_N(D'/K\otimes_\Qp E)\geq t_H(D'/K\otimes_\Qp E)$ holds
for every sub-object $D'$ of $D$, and equality holds for $D$.
\item[(iv)] $t_N(D')\geq t_H(D')$ holds for every sub-object $D'$ of $D$, and equality holds for $D$.
\end{enumerate}
\end{proposition}

\begin{proof}
It is Fontaine and Fontaine-Colmez's theorem \cite{Fo88a}\cite{Fo88b}\cite{CF00}
that (i)$\Longleftrightarrow$(ii). See discussion in \cite[Section 3]{BS07}.
It is well known that (ii)$\Longleftrightarrow$(iii) see for example
\cite{Fo79}. The equality (iii)$\Longleftrightarrow$(iv) is clear from remark above this proposition.
\end{proof}

\begin{remark}
For two dimensional case, suppose $D$ is of embedded
Hodge polygon $\HP_{\Z/a\Z}(D) =(k_j,0)_{j\in\Z/a\Z}$, then
$\HP^\sigma(D)=(\frac{1}{a}\sum_{j\in\Z/a\Z}k_j, 0)=(\bar{k}_1,0)$.
If $D$ is weakly admissible, then $\NP^\sigma(D)=(\bar{k}_1-\alpha,\alpha)$
for some $\alpha\geq 0$ since it lies over $\HP^\sigma(D)$.
In Breuil-Schneider's notation, the Hodge polygon is of slopes $0,\sum_{j\in\Z/a\Z}k_j$
and the Newton polygon is of slopes $a\alpha, \sum_{j\in\Z/a\Z}k_j - a\alpha$.
\end{remark}

\subsection{Irreducibility}
\label{S:3.3}

Given any Hodge polygons
$\k=(\k_j)_{j\in\Z/a\Z}$ and $\k'=(\k'_j)_{j\in\Z/a\Z}$,
if $\k'_j$ is a sub-polygon of $\k_j$ for every $j$, then
we denote it by $\k'\subseteq \k$.
Let $\Theta$ be the bijection in Theorem \ref{T:A-analog}, we
assume that $V=\Theta(A)$ where $A=(A_j)_j \in \GL_d(\cO_E)^a$.
For any Hodge polygon
$\k$ and $A\in \GL_d(\cO_E)^a$ let
$P_{(j),A}:=\prod_{l=0}^{a-1}A_{l+j+1}\Delta_{\k_{l+j+1}}
=(A_{j+1}\Delta_{\k_{j+1}})(A_{j+2}\Delta_{\k_{j+2}})\cdots (A_j\Delta_{\k_j})$,
where $\Delta_{\k_j}=\diag(p^{k_{j1}},\ldots,p^{k_{jd}})$.

In the bijection $\Theta: \GL_d(\cO_E)^a/\sim_\k
\longrightarrow \Rep_{\cris/E}^{\k}(G_K)$ (see Theorem \ref{T:A-analog}) we identify the pre-image of
those irreducible representations in the following proposition.
Its application in this paper lies in Theorem \ref{T:irreducible}.

\begin{proposition}\label{P:irreducibility-2}
Let $V\in \Rep_{\cris/E}^\k (G_K)$ and
$\Theta^{-1}(V)=A\in \GL_d(\cO_E)^a$.
Then the followings are equivalent to each other
\begin{enumerate}
\item[(i)] $V$ is irreducible.
\item[(ii)] There do not exist $\k'\subsetneq\k$ of length $d'$,
$A'\in\GL_{d'}(\cO_E)^a$, and $C=(C_j)_j\in M_{d\times d'}(\cO_E)^a$
with $\rank(C_j)=d'$, such that
\begin{eqnarray}\label{E:irreducible}
C_{j-1}  &=&  (A_j \Delta_{\k_j}) C_j (A'_j \Delta_{\k'_j})^{-1}
\end{eqnarray}
for all $j\in\Z/a\Z$.
\item[(iii)]
There do not exist $\k'\subsetneq \k$ of length $d'$,
$A'\in \GL_{d'}(\cO_E)^a$, and $C=(C_j)_j\in M_{d\times d'}(\cO_E)^a$
with $\rank(C_j)=d'$, such that any one of equations (\ref{E:irreducible})
replaced by
\begin{eqnarray}\label{E:irreducible-2}
C_j  &=&  P_{(j),A} \cdot C_j \cdot (P_{(j),A'})^{-1}
\end{eqnarray}
\end{enumerate}
\end{proposition}
\begin{proof}
The equivalence between parts i) and ii)
follows the same argument as in that of
Proposition \ref{P:irreducibility},
combined with our results in Section \ref{S:3}
for $\Rep_{\cris/E}^\k(G_K)$,
hence we omit it here.
The equivalence of parts ii) and iii) can be seen by
iteration of (\ref{E:irreducible}) $a$ many times and
arrive at (\ref{E:irreducible-2}).
\end{proof}

\section{Constructing embedded Wach modules}
\label{S:4}

We prove Theorems B and C in this section.
Our tool is strongly divisible lattices in Section \ref{S:3},
and theory of Wach modules developed by
Fontaine, Wach \cite{Wa96}, Colmez \cite{Co99}, and Berger \cite{BE04}.
It should be evident to the reader that we benefit especially
from recent papers and exposition by Berger \cite{BE04}
and Berger-Breuil \cite{BB04}.

\subsection{($\varphi,\gamma$)-equations and basic arithmetic}
\label{S:4.1}

This subsection contains some essential yet rather technical results
we shall need in the rest of paper.
For this subsection (only) $K$ and $E$ are complete
discrete valuation fields of characteristic $0$ with perfect
residue field of characteristic $p$. Write $\mu_{p^\infty}$ for
the set of all $p$-power roots of unity in $\bar{K}$ and let
$\Gamma_K=\Gal(K(\mu_{p^\infty})/K)$. Let $\chi$ be the $p$-adic cyclotomic
character $G_K\rightarrow \Zp^*$, which in fact factors through
$\Gamma_K$, defined by $g(\zeta)=\zeta^{\chi(g)}$ for any $g\in
G_K$ and $\zeta\in\mu_{p^\infty}$. Let $\pi$ be a variable. For
any $\gamma\in \Gamma_K$, let $\varphi, \gamma$ be $E$-linear
operators on $K[[\pi]]\otimes_\Qp E$ commuting with each other such
that $\varphi(\pi) = (1+\pi)^p-1$ and $\gamma(\pi)=
(1+\pi)^{\chi(\gamma)} -1$. Let
$\varphi$ acts $\sigma$-semilinearly on $K$, where $\sigma$ is the absolute
Frobenius on $K$.

Write $q=\varphi(\pi)/\pi$. Notice that $q^{\gamma-1} \in 1+\pi\Zp[[\pi]]$.
For any integer $b\geq 1$ let $\lambda_b:= \prod_{n=0}^{\infty} \frac{\varphi^{bn}(q)}{p}$. It is not hard to see that $\lambda_b \in 1+\pi\Qp[[\pi]]$. Even though $\lambda_b$ does not lie in $\Zp[[\pi]]$ for $b>1$,
we have $\lambda_b^{\gamma-1} \in 1+\pi\Zp[[\pi]]$ (since
$\lambda_b^{\gamma-1} = \prod_{n=0}^{\infty}\varphi^{bn}(q^{\gamma-1})$).
Recall definition of $\cR_{c,E}$ from the end of Section \ref{S:1}.
It is clear $q/p = 1+\pi+\cdots + p^{-1}\pi^{p-1} \in \cR_{p-1,\Qp}$.
In this paper our convention is
$g^{\varphi-1} = \frac{\varphi(g)}{g}$ for any $g$.

\begin{lemma}\label{L:banach-solution1}
\begin{enumerate}
\item[(i)]
For a power series $h\in 1+\pi E[[\pi]]$ there is a
unique $g\in 1+\pi E[[\pi]]$ such that $g^{\varphi^b-1} = h$ for
any $b\geq 1$.
Let $c$ be any nonzero number, then we have $h\in \cR_{c,E}$ if and only if
$g\in \cR_{c,E}$.
(In particular, $h$ is integral if and only if $g$ is integral.)

\item[(ii)]
Let notation be as in Part (i).
Let $f(\varphi,\gamma)\in \Z[\varphi,\gamma]$.
Then $g^{\varphi^b-1} = (q/p)^{f(\varphi,\gamma)}$
has a unique solution $g\in 1+\pi E[[\pi]]$
and $g=\lambda_b^{-f(\varphi,\gamma)}$.
Moreover, we have $\lambda_b \in \cR_{p-1,\Qp}$.

\item[(iii)]
Part (i) holds if $\varphi^b-1$ is replaces by $\varphi^b+1$.
The unique solution to $g^{\varphi^b+1}=(q/p)^{f(\varphi,\gamma)}$
is $g=\lambda_{2b}^{-f(\varphi,\gamma)(\varphi^b-1)}.$

\item[(iv)]
If $\gamma\in \Gamma_K$ has infinite order, then for any $f\in 1+\pi E[[\pi]]$,
there exists a unique $h\in 1+\pi E[[\pi]]$ satisfying $h^{\gamma-1} = f$.
\end{enumerate}
\end{lemma}
\begin{proof}
(i)
Let $g=\sum_{n=0}^{\infty} g_n
\pi^n$ with $g_0=1$. We shall prove our assertion by induction on
$n$.
Suppose we have $g_1,\ldots,g_{n-1}\in E$ such that
$g^{\varphi^b} \equiv h g\bmod \pi^n$. We claim that there exists a
unique $g_n\in E$ such that $g^{\varphi^b} \equiv h g \bmod
\pi^{n+1}$. Comparing the coefficients of $\pi^n$ we see that the
latter congruence is equivalent to
\begin{eqnarray}\label{E:g_n}
g_n &=& g_n^{\varphi^b}p^{nb} + V_n
\end{eqnarray}
where $V_n\in \Zp[h_1,\ldots,h_n,g_1^{\varphi^b},\ldots,g_{n-1}^{\varphi^b}]$.
Let $\psi:E\rightarrow E$ be defined by sending $g_n$ to
$g_n^{\varphi^b}p^{nb}+V_n$. It is clear that $\psi$ is a contraction
map on the complete $p$-adic field $E$, then by Banach contraction
mapping theorem there exists a unique $g_n\in E$ satisfying
$\psi(g_n)=g_n$. This proves the existence and uniqueness of solution $g$.
Now suppose $h\in \cR_{c,E}$, we shall show
$\ord_p g_n \geq - \frac{n}{c}$ for all $n\geq 0$.
This clearly holds for $n=0$.
By induction we may suppose it holds for all $g_1,\ldots,g_{n-1}$.
Recall, in (\ref{E:g_n}) above, $V_n = - h_n + \sum_{t<n}g_t^{\varphi^b}a_t
- \sum_{s+t=n, t<n}h_s g_t$ for some $a_t\in \Zp$. So we have
$\ord_p V_n \geq \min(\ord_p h_n, \min_{t<n}g_t, \min_{s+t=n,t<n}\ord_p h_s g_t)
\geq -\frac{n}{c}$ by our inductive hypothesis.
Thus
$\ord_p g_n = \ord_p (g_n -g_n^{\varphi^b}p^{nb}) = \ord_p V_n \geq -\frac{n}{c}$.
This proves our claim that $g\in \cR_{c,E}$.
The converse direction follows from a similar only easier
and we shall omit its proof here.

(ii) The first assertion follows from a direct verification.
Since $q/p \in \cR_{p-1,\Qp}$ we have $\lambda_b^{-1} \in \cR_{p-1,\Qp}$.
But it is clear that $\lambda_b$ is a unit in $\cR_{p-1,\Qp}$ and
hence $\lambda_b \in \cR_{p-1,\Qp}$.

(iii) These follows from parts (i) and (ii) where we proved
the existence and uniqueness of $g$ to
$g^{\varphi^{2b}-1}=h^{\varphi^b-1}$, that is $g^{\varphi^b+1}=h$.

(iv)
Since $E[[\pi]]$ is $\pi$-adically complete, it suffices to
show there is a unique solution $h\in 1+\pi E[[\pi]]$ such that
\begin{eqnarray}\label{E:hh}
h^{\gamma-1} - f  & \equiv & 0
\bmod \pi^n \quad \mbox{for all $n\geq 1$}.
\end{eqnarray}
It is clear for $n=1$ and suppose it holds for $n$.
Write $h_n:= (h \bmod \pi^{n+1}) = h_{n-1}+ \pi^n H_n$ with $H_n\in E$
and $f_n := \sum_{\ell=0}^{n}F_\ell\pi^\ell = (f\bmod \pi^{n+1})$.
By our inductive hypothesis, we should write
$h_{n-1}^{\gamma-1} - f_{n-1} = \pi^n R_{n-1}$
for some $R_{n-1} \in E[[\pi]]$.
Then (\ref{E:hh}) is equivalent to
\begin{eqnarray}\label{E:H_n}
((\gamma(\pi)/\pi)^n-1)H_n &\equiv& F_n- R_{n-1} \bmod \pi.
\end{eqnarray}
But $\gamma(\pi)/\pi \equiv \chi(\gamma) \bmod \pi$,  since
$\gamma$ is of infinite order we have $\chi(\gamma)^n -1 \neq 0$
in $\Zp$. Therefore, (\ref{E:H_n}) has a solution in $E$ and this
proves our claim.
\end{proof}

As an application of the above
lemma, the following key lemma
will be used in Proposition \ref{P:3.1.2-analog},
when we come to consider
$2$-dimensional representations.

\begin{lemma}\label{L:solution1}
\begin{enumerate}
\item[(i)]
The following system of equations for all $j\in\Z/a\Z$
\begin{equation}\label{E:s_j}
\left\{
\begin{array}{llll}
&s_{j-1}  = t_j^\varphi, \quad & t_{j-1} = q^{(1-\gamma)k_j}s_j^\varphi & \mbox{if
$j\in \cA$}\\
&s_{j-1} = q^{(1-\gamma)k_j}s_j^{\varphi},\quad
&t_{j-1} = t_j^\varphi & \mbox{if $j\in\cB$}.
\end{array}
\right.
\end{equation}
has a unique solution set $\cQ:=(s_0,\ldots,s_{a-1},t_0,\ldots,t_{a-1})$
with each $s_j,t_j \in 1+\pi\Zp[[\pi]]$:
$$s_j =\lambda_b^{(1-\gamma)f_j(\varphi)},\quad
t_j = \lambda_b^{(1-\gamma)g_j(\varphi)}$$
for $b=a$ or $2a$ and for some polynomials
$f_j(\varphi), g_j(\varphi)$ in variable $\varphi$
and with coefficients equal to some algebraic combinations of $k_0,\ldots,k_{a-1}$.

\item[(ii)]
We have $b=a$ if and only if $\#\cA$ is even;
and $b=2a$ if and only
if $\# \cA$ is odd.

\item[(iii)]
Let $\cS$ be the solution set of $(h_0,\ldots,h_{a-1})$ in $E[[\pi]]^a$
to the following system of $\gamma$-equations for all $j\in\Z/a\Z$
\begin{equation}\label{E:h_j}
h_j^{\gamma-1} = \frac{s_{j-1}}{t_{j-1}}.
\end{equation}
Then we have
$$\cS \cap (1+\pi E[[\pi]])= (\lambda_b^{g_{j-1}(\varphi)-f_{j-1}(\varphi)})_{j\in\Z/a\Z}
\in \cR_{p-1,\Qp}^a$$
and $\cS = (E^*\lambda_b^{g_{j-1}(\varphi)-f_{j-1}(\varphi)})_{j\in\Z/a\Z}$.
\end{enumerate}
\end{lemma}

\begin{proof}
(i) and (ii).
In the solution set $\cQ$, if any two entries $A, B$ satisfy an
equation of the form $A^\varphi = u B$ for some unit $u\in
1+\pi\Zp[[\pi]]$ of the form $q^{(1-\gamma)f(\varphi)}$ for some
$f(\varphi) \in \Z[\varphi]$, then we mark a directed path from
$A$ to $B$ in $\cQ$, denoted by \fbox{A}$\rightarrow$ \fbox{B}.
For any $j\in \cA$ we have \fbox{$t_j$}$\rightarrow$
\fbox{$s_{j-1}$} and \fbox{$s_j$} $\rightarrow$ \fbox{$t_{j-1}$}; For
any $j\in \cB$ we have \fbox{$t_j$}$\rightarrow$ \fbox{$t_{j-1}$} and
\fbox{$s_j$}$\rightarrow$ \fbox{$s_{j-1}$}. This shows that the
directed graph of the solution set $\cQ$ consists exactly two
disjoint $a$-cycles or exactly one $2a$-cycle: The walk initiated
from $s_j$ will travel to $s_{j-1}$ or $t_{j-1}$, and so on, this
walk will return (for the first time) to $s_j$ or $t_j$ after $a$
steps. If it returns to $s_j$ then it forms an $a$-cycle. By
symmetry, its complement is an $a$-cycle containing $t_j$; On the
other hand, if it returns to $t_j$ after $a$ steps, then by
symmetry again, this walk will return to $s_j$ after $a$ steps,
hence it forms an $2a$-cycle. Every directed path obtained in this
manner initiated from $s_j$ (or $t_j$) indicates an equation of
the form $s_j^{\varphi^b-1} = q^{(\gamma-1)f_j(\varphi)}$ (or
$t_j^{\varphi^b-1}=q^{(\gamma-1)g_j(\varphi)}$) for some
polynomials $f_j(\varphi), g_j(\varphi)\in \Z[\varphi]$ whose
coefficients are algebraic combinations of $k_0,\ldots,k_{a-1}$.
Note that $q^{\gamma-1}=(q/p)^{\gamma-1}$, we may apply Lemma
\ref{L:banach-solution1} (ii) and obtain unique solution $s_j$ and
$t_j$ both in $1+\pi\Zp[[\pi]]$. In particular, $s_j=
\lambda_b^{(1-\gamma)f_j(\varphi)}$ (and $t_j =
\lambda_b^{(1-\gamma)g_j(\varphi)}$ respectively) for the system
$s_j^{\varphi^b-1}=(q/p)^{(\gamma-1)f_j(\varphi)}$ (and
$t_j^{\varphi^b-1} = (q/p)^{(\gamma-1)g_j(\varphi)}$
respectively).

(iii).
By Lemma \ref{L:banach-solution1}(iv), there is a unique solution $(h_j)_j$
to $(\ref{E:h_j})$;
Using notations from the first part of this lemma it is easy to check that
$h_j = \lambda_b^{(g_{j-1}-f_{j-1})(\varphi)}$.
Note that $h_j\in \cR_{p-1,\Qp}$ since $\lambda_b\in \cR_{p-1,\Qp}$
and the ring $\cR_{p-1,\Qp}$
is closed under $\varphi$.
\end{proof}

Below we prove some basic arithmetic of Banach contraction maps,
which we shall use Proposition \ref{P:continuity}.

\begin{lemma}\label{L:continuity}
1) Let $\fG,\fG'$ be two $p$-adic contraction maps on
$M_{d\times d}(\cO_E)$. Let $Z,Z'$ be their fixed points respectively. Suppose
that for all $x\in M_{d\times d}(\cO_E)$ we have $\ord_p(\fG(x)-\fG'(x)) \geq m\geq 0$
then we have $\ord_p(Z-Z') \geq m$.

2) For any $M,M' \in \GL_d(\cO_E)$ we have
$\ord_p(M^{-1}-M'^{-1}) = \ord_p(M-M')$.

3) Let $M_j,M'_j\in \GL_d(\cO_E)$ and $B_j\in M_{d\times d}(\cO_E)$ for all $j=1,\ldots,n$
then
\begin{eqnarray*}
\ord_p(\prod_j M_j B_j - \prod_j M'_jB_j) &\geq & \min_j(\ord_p(M_j-M'_j)),\\
\ord_p(\prod_j B_jM_j^{-1} - \prod_j B_j{M'}_j^{-1})
 &\geq &\min_j(\ord_p(M_j-M'_j)).
\end{eqnarray*}
\end{lemma}
\begin{proof}
1) By setting $x=Z$ our hypothesis gives
$\ord_p(Z-\fG'(Z))\geq m$.
Then $\ord_p({\fG'}^n(Z)-{\fG'}^{n+1}(Z)) \geq \ord_p(Z-\fG'(Z))\geq m$
for all $n\geq 0$. So
$\ord_p(Z-{\fG'}^n(Z))\geq m$ for all $n\geq 0$.
Hence $\ord_p(Z-Z')\geq m$ as
$Z'=\lim_{n\rightarrow \infty}{\fG'}^n(Z)$.

2) One only has to observe that
${M'}^{-1} - M^{-1} = {M'}^{-1}(M-M')M^{-1}$ and then
take $p$-adic order both sides.

3) Use induction argument on $j\geq 1$, the first inequality easily follows.
The second follows by a similar argument and part (2).
\end{proof}

\subsection{Embedded Wach liftings from $\MF{ad}(K\otimes_\Qp E)$}
\label{S:4.2}

The goal of this subsection is to prove Proposition \ref{P:unique-split},
which will be the key ingredient in proving Theorems B and C in Section
\ref{S:4.3}.
We first take the liberty to recall without proof some fundamental notions, many of which can be found in \cite{Wa96},\cite{Co99}, and \cite{BE04}.
Retaining notations from Section \ref{S:3}, let $D=\prod_{j\in\Z/a\Z} D_j$ be a weakly admissible filtered $\varphi$-module in $\MF{ad,\k}(K\otimes_\Qp E)$.
Let $L=\prod_j L_j$ be an embedded strongly divisible $\varphi$-lattice of $D$.
Assume $k_{jd}=0$ for all $j$ and write $\bar{k}_1:=\frac{1}{a}\sum_{j\in\Z/a\Z}k_{j1}
=\MaxSlope(\HP^\sigma(V/K\otimes_\Qp E))$.
Fix $\Delta_{\k_j}=\diag(p^{k_{j1}},\ldots,p^{k_{jd}})$.
Write $P_j=\Mat(\varphi|_{L_j})=A_j\Delta_{\k_j}$ for some
$A_j\in\GL_d(\cO_E)$, with respect to basis adapted to filtration.
Then $\Mat(\varphi^a|_{L_j})= P_{(j)} :=
\prod_{\ell=0}^{a-1}P_{\ell+j+1}$ with the subindex $\ell$
ranging from $0$ to $a-1$ (necessarily in this order) in $\Z/a\Z$ (since $\varphi$ acts on $E$ as identity).
Recall from Theorem \ref{T:A-analog} that
every $(A_j)_j \in \GL_d(\cO_E)^a$ determines a $V\in \Rep_{\cris}^\k(K\otimes_\Qp E)$.
Recall $\varphi$-action on $\akplus$ is given by $\varphi(\pi)=(1+\pi)^p-1$.
Let $q:=\varphi(\pi)/\pi=((1+\pi)^p-1)/\pi$.
Fix $\hat\Delta_{\k_j} =\diag(q^{k_{j1}},\ldots,q^{k_{jd}}) \in M_{d\times d}(\Zp[\pi])$.

Recall that an \'etale $(\varphi,\Gamma)$-module $M$ over
$\A_K\otimes_\Zp \cO_E$ is a $\A_K\otimes_\Zp \cO_E$-module
of finite rank, with continuous and commuting $\A_K$-semilinear
and $\cO_E$-linear $\varphi$- and $\Gamma_K$-actions such that
$\varphi(M)$ generates $M$ over $\A_K$. Fontaine's functor $\D:
T\rightarrow \D(T)$ associates any de Rham $\cO_E$-linear
representation $T$ of $G_K$ to an \'etale
$(\varphi,\Gamma)$-module over $\A_K\otimes_\Zp \cO_E$. By
inverting $p$, one also gets an equivalence of categories between
the category of $E$-linear representations and the category of
\'etale $(\varphi,\Gamma)$-modules over $\B_K\otimes_\Qp E$.
Let $\MF{}(\bkplus\otimes_\Qp E)$ be the category of filtered
$\varphi$-modules over $\bkplus\otimes_\Qp E$. Any object $N$
in $\MF{}(\bkplus\otimes_\Qp E)$ may be written as $N=\prod_j N_j$
according to the split $\bkplus\otimes_\Qp E =\prod_j \beplus$.

For any $(D,\Fil, \varphi)$ in $\MF{ad,\k}(K\otimes_\Qp E)$,
we define an {\em embedded Wach module} $\N$ in
$\MF{}(\bkplus\otimes_\Qp E)$ attached to it (if such
$\N$ exists) as a free $\bkplus\otimes_\Qp E$-module.
Write $\N=\prod_{j\in \Z/a\Z}\N_j$
where $\N_j$ is a free $\beplus$-module
of rank $d$ such that the following (first) 6 conditions are satisfied
\begin{enumerate}
\item $\Fil^i \N_j := \{x\in \N_j|\varphi(x)\in q^i \N_{j-1}\}$,

\item we have isomorphism $\theta_j: \N_j/\pi
\N_j \rightarrow D_j$ as filtered $E$-modules. Write $\varphi_j:=\varphi|_{\N_j}$
also for induced map on $\N_j/\pi \N_j$, we have
$\varphi_j\theta_j = \theta_{j-1} \varphi_j$,

\item every $\gamma\in \Gamma_K$ acts continuously
on $\prod_j\N_j$, preserves $\N_j$ and
$\varphi\gamma = \gamma \varphi$,

\item $\Gamma_K$ acts trivially on $\N_j/\pi
\N_j$ for every $j$,

\item
$\HP(\varphi|_{\N_j/\beplus})=\diag(q^{k_{j1}},\ldots,q^{k_{jd}})$,

\item The $(\varphi,\Gamma)$-module generated by $\N$ is \'etale.

\item[(7*)] $\N_j$ contains a $\aeplus$-lattice $\N(T)_j$ that is
$\Gamma_K$-stable and $\varphi(\N(T)_j)\subseteq \N(T)_{j-1}$.
\end{enumerate}
We shall call this process of constructing Wach modules from
weakly admissible filtered $\varphi$-modules {\em embedded Wach
lifting}.
They form a subcategory $\MF{ad}(\bkplus\otimes_\Qp E)$ in
$\MF{}(\bkplus\otimes_\Qp E)$.
The $\akplus\otimes_\Zp\cO_E$-module
$\N(T):=\prod_{j=0}^{a-1} \N(T)_j$ in (7*) above is called {\em
integral embedded Wach module}. It is a full lattice in $\N$.
Below we shall prove two twchnical lemmas in preparation
for Proposition \ref{P:unique-split}.

\begin{lemma}\label{L:solution}
(i) Suppose $s\in M_{d\times d}(\cO_E[[\pi]])$
such that $q^{a\bar{k}_1}s^{-1}\in M_{d\times d}(\cO_E[[\pi]])$.
Set $f(s,b) = b-sb^{\varphi^a} s^{-\gamma}$.
If $(b_{n-1} \bmod \pi^n)\in M_{d\times d}(\cO_E[\pi])$ such that
$f(s,b_{n-1}) \in \pi^n M_{d\times d}(\cO_E[[\pi]])$ for some $n >
\bar{k}_1$,
then there exists a unique $b\in M_{d\times d}(\cO_E[[\pi]])$
with $b\equiv b_{n-1} \bmod \pi^n$ such that $f(s,b)=0$.

(ii) Suppose $k_{j1}\geq \ldots \geq k_{j,d-1}>k_{jd}=0$
and suppose $s=\hP_{(0)}$. Then the above statement holds for
$n\geq \bar{k}_1$.
\end{lemma}
\begin{proof}
(i)
Write $b_n =\sum_{i=0}^{n}B_i \pi^i$.
By hypothesis, we write $f(s,b_{n-1}) = \pi^n R_n$ for some
$R_n\in M_{d\times d}(\cO_E[[\pi]])$. Then we have
\begin{eqnarray*}
f(s,b_n)
&=& f(s,b_{n-1})+\pi^n(B_n-q^{an} sB_n^{\varphi^a}s^{-\gamma})\\
&=& \pi^n(B_n-q^{a(n-\bar{k}_1)}sB_n^{\varphi^a}q^{a\bar{k}_1}s^{-\gamma}+ R_n).
\end{eqnarray*}
Notice that $q\equiv p\bmod \pi$, and thus $p^{a\bar{k}_1}s^{-1}(0) \in M_{d\times d}(\cO_E)$.
We claim the following equation has a unique solution $B_n\in M_{d\times d}(\cO_E)$:
\begin{eqnarray}\label{E:B_n}
B_n &=& p^{a(n-\bar{k}_1)}s(0) B_n^{\varphi^a}p^{a\bar{k}_1}s(0)^{-1}  - R_n(0).
\end{eqnarray}
For $n>\bar{k}_1$
the map sending any matrix $B_n$ to the right-hand-side of the
above equation is clearly a $p$-adic contraction map on $M_{d\times d}(\cO_E)$,
hence our above claim follows from the Banach mapping theorem.
Since $q^{a\bar{k}_1}s^{-1}\in M_{d\times d}(\cO_E[[\pi]])$
and $q^{\gamma-1} \in 1+\pi \Zp[[\pi]]$,
we have $f(s,b_n)/\pi^n \in \pi M_{d\times d}(\cO_E[[\pi]])$
and therefore we have $ R_{n+1}\in M_{d\times d}(\cO_E[[\pi]])$
as claimed in our induction.
This finishes the proof of part (i).

(ii). Suppose $n=\bar{k}_1$ (only when $\bar{k}_1$ is an integer).
Since $s(0) \equiv (\0^T|\star)\bmod p$ and
$p^{a\bar{k}_1}s(0)^{-1}\equiv (\frac{\star}{\0}) \bmod p$, one
can derive via a calculation that the map $\fI: B_n\rightarrow B_n
- s(0) B_n^{\varphi^a}p^{a\bar{k}_1}s(0)^{-1} +R_n(0)$ induces an
invertible map on $M_{d\times d}(\cO_E/\m_E)$. Hence $f(s,b)=0$ has a unique
solution for $b$ modulo $(p,\pi)$. By Hensel's lemma we may lift
it to a unique solution to $b\in M_{d\times d}(\cO)E[[\pi]])$.
\end{proof}

\begin{lemma}\label{L:commute}
Let $N=\prod_j N_j \in \MF{}(\bkplus \otimes_\Qp E)$ be given by
$\Mat(\varphi|_{N_j})=\hP_j$ and $\Mat(\gamma|_{N_j}) = \hG_{\gamma,j}$.
Then $\Mat(\varphi^a|_{N_j})=\hP_{(j)}:=\prod_{l=0}^{a-1}\varphi^{l}(\hP_{l+j+1})
=\hP_{j+1}\cdot\varphi(\hP_{j+2})\cdot\cdots\cdot
\varphi^{a-1}(\hP_j)$
where subindex lies in $\Z/a\Z$.
The following statements are equivalent:
\begin{enumerate}

\item we have $\varphi\gamma = \gamma \varphi$ on $\N$ for every $\gamma\in \Gamma_K$.

\item we have
\begin{eqnarray*}
&&\left(
\begin{array}{cccc}
  0 &  &  & \hat{P}_0 \\
  \hat{P}_1 &  &  &  \\
   & \ddots &  &  \\
  0 &  & \hat{P}_{a-1} & 0 \\
\end{array}
\right)
\left(
\begin{array}{cccc}
  \hG_{\gamma,1}^\varphi &  &  & 0 \\
   &\hG_{\gamma,2}^\varphi  &  &  \\
   & & \ddots &  \\
  0 &  &  & \hG_{\gamma,0}^\varphi \\
\end{array}
\right)\\
&&\quad\quad =
\left(
\begin{array}{cccc}
  \hG_{\gamma,a-1} &  &  & 0 \\
   &\hG_{\gamma,0}   &  &  \\
  & & \ddots &  \\
  0 &  &  & \hG_{\gamma,a-2} \\
\end{array}
\right)
\left(
\begin{array}{cccc}
  0 &  &  & \hP_0^\gamma \\
  \hP_1^\gamma &  &  &  \\
   & \ddots &  &  \\
  0 &  & \hP_{a-1}^\gamma & 0 \\
\end{array}
\right)
\end{eqnarray*}
on $\prod_{j=0}^{a-1} \N_j$.

\item
$\hG_{\gamma,j-1}\hP_j^\gamma =\hP_j\hG_{\gamma,j}^\varphi$ for
$j\in\Z/a\Z$.

\item
Replace an equation in Part (3)
$\hG_{\gamma,j-1}\hP_j^\gamma =\hP_j\hG_{\gamma,j}^\varphi$
by the following equation
\begin{eqnarray}\label{E:PP}
\hG_{\gamma,j} \hP_{(j)}^{\gamma}&=&\hP_{(j)}\hG_{\gamma,j}^{\varphi^a}
\end{eqnarray}
\end{enumerate}
\end{lemma}
\begin{proof}
We will only discuss the proof of (4)$\Longrightarrow$(3) as the rest are
straightforward computations.
Suppose $\hP_j$'s are given for all $j\in\Z/a\Z$.
Let $\hG_{\gamma,j}$ be a solution to
the equation (\ref{E:PP}) for a particular $j$,
then we can recover all $\hG_{\gamma,j}$ with $j\in\Z/a\Z$ by the equations
in Part (3). These $\hG_{\gamma,j}$'s clearly are solutions to the system.
\end{proof}

\begin{remark}\label{R:Berger-Wach}
We can consider the above embedded Wach lifting
$\Upsilon: \MF{ad}(K\otimes_\Qp E) \longrightarrow \MF{ad}(\bkplus\otimes_\Qp E)$, via the bijection of Theorem \ref{T:A-analog}, as an explicit
constructing map $\Upsilon: \GL_d(\cO_E)^a\longrightarrow \GL_d(\cO_E[[\pi]])^a$
defined by $\Upsilon((A_j)_j) = (\hat{A}_j)_j$ such that
the matrices $\hat{A}_j \in A_j+\pi M_{d\times d}(\cO_E[[\pi]])$ satisfy
certain commuting property with $\hat{G}_j$'s (as in the above lemma).
\end{remark}

\begin{proposition}\label{P:unique-split}
Let $n>\bar{k}_1$
(and $n\geq \bar{k}_1$ if
$k_{j1}\geq \ldots \geq k_{j,d-1}> k_{j,d}=0$).
Suppose $\hA_j\in \GL_d(\cO_E[[\pi]])$ lifts given $A_j$
and suppose $\hG_{\gamma,j,n-1}$
lies in $\Id + \pi M_{d\times d}(\cO_E[[\pi]])$ such that
$\hG_{\gamma,j-1,n-1}-\hP_j\hG_{\gamma,j,n-1}^{\varphi}
\hP_j^{-\gamma}\in \pi^n M_{d\times d}(\cO_E[[\pi]])$,
then there exist unique
$\hG_{\gamma,j} \equiv \hG_{\gamma,j,n-1}\bmod \pi^n$
in $M_{d\times d}(\cO_E[[\pi]])$
such that for every $j\in\Z/a\Z$ we have
$\hG_{\gamma,j-1}\hP_j^{\gamma} = \hP_j \hG_{\gamma,j}^\varphi$.
\end{proposition}

\begin{proof}
Apply Lemma \ref{L:solution} to $s=\hP_{(0)}$ we have a unique
solution $\hG_{\gamma,0}\equiv \hG_{\gamma,0,n-1}\bmod
\pi^n$ as desired. We obtain the rest $\hG_{\gamma,j}$
with $j\in \Z/a\Z$ by the simple formula $\hG_{\gamma,j-1} = \hP_j
\hG_{\gamma,j}^\varphi \hP_j^{-\gamma}$ where
$\hP_j=\hA_j\hat\Delta_{\k_j}$. By Lemma \ref{L:commute}, especially the
equivalence between part (3) and part (4), these
$\hG_{\gamma,j}$ are precisely the unique solutions to the given
system.
\end{proof}

\begin{remark}\label{R:addX}
If in the hypotheses of Lemma \ref{L:solution} and Proposition \ref{P:unique-split}
the ring $\cO_E$ is replaced by any unit disc in any $p$-adically complete space,
for instance $\cO_E[[\vX]]$, an analogous statement still holds.
We shall use this more general version of the proposition in the proof
of Proposition \ref{P:continuity2}.
\end{remark}

\subsection{Continuity of embedded Wach modules}
\label{S:4.3}

Pick $m=\ink+1$ (and $m=\fk$ if
$k_{j1}\geq \ldots \geq k_{j,d-1}>k_{jd}=0$).
Given $A_j\in \GL_d(\cO_E)$,  suppose we may obtain
$\hA_j \equiv A_j\bmod \pi$
and $\hG_{\gamma,j,m-1}$ satisfy the hypothesis
$\hG_{\gamma,j-1,m-1} - \hP_j \hG_{\gamma,j,m-1}^\varphi
\hP_j^{-\gamma} \in\pi^{m}M_{d\times d}(\cO_E[[\pi]])$,
then by Proposition \ref{P:unique-split}
we have a unique $\hG_{\gamma,j} \in 1+\pi M_{d\times d}(\cO_E[[\pi]])$ lifting
$\hG_{\gamma,j,m-1}$ such that
the system in Lemma \ref{L:commute}(3) is satisfied.
This Wach lifting process, for any given $A_j$,
yields a map $\fW: \Gamma_K\rightarrow \prod_{j\in \Z/a\Z}M_{d\times d}(\cO_E[[\pi]])$
given by $\gamma\mapsto (\hG_{\gamma,j})_{j\in\Z/a\Z}$.

\begin{proposition}\label{P:continuity}
Let notation be as above paragraph.
The map $\fW$ defined above is continuous ($p$-adically).
\end{proposition}
\begin{proof}
Fix $A_j\in \GL_d(\cO_E)$.
In the above defined Wach lifting process it suffices
to show for a fixed $\gamma\in\Gamma_K$
the map $(\hA_j,\hG_{\gamma,j,m-1})\mapsto \hG_{\gamma,j}$
is continuous in the sense that
for $A'_j\in \GL_d(\cO_E)$ we have
$$\ord_p(\hG_{\gamma,j}- \hG'_{\gamma,j})
\geq \min(\ord_p(\hA_j-\hA'_j),
\ord_p(\hG_{\gamma,j,m-1}-\hG'_{\gamma,j,m-1})).$$ Consider the
contraction map on $M_{d\times d}(\cO_E)$ for given $A_j$ defined by $\fG:
B_m\mapsto p^{a(m-\bar{k}_1)} P_{(0)} B_m^{\varphi^a}
p^{a\bar{k}_1}P_{(0)}^{-1} -R_m(0)$ as seen in the proof of Lemma
\ref{L:solution} (in particular, see (\ref{E:B_n})). Let $\fG'$
denote the map corresponding to given $A'_j$. Recall $P_{(0)} =
\prod_{j}A_j\Delta_{\k_j}$ (how subindex $j$ ranges does not affect the
proof below so we do not deliberate), then $P'_{(0)} = \prod_j
A'_j\Delta_{\k_j}$. We write $p^{a\bar{k}_1}P_{(0)}^{-1} =
\prod_{j}\Delta_{\k_j}^*A_j^{-1}$ where
$\Delta_{\k_j}^*=p^{k_{j1}}\Delta_{\k_j}^{-1}$ is diagonal in $M_{d\times d}(\cO_E)$.
We observe easily that $\ord_p (\fG(B_m)-\fG'(B_m))$
\begin{eqnarray*}
 &=& a(m-\bar{k}_1)+
\ord_p(P_{(0)} B_m^{\varphi^a} p^{a\bar{k}_1}P_{(0)}^{-1}
      -P'_{(0)} B_m^{\varphi^a} p^{a\bar{k}_1}{P'}_{(0)}^{-1}
)\\
&\geq& \min(\ord_p(\prod_j
A_j\Delta_{\k_j}-\prod_jA'_j\Delta_{\k_j}),\ord_p(\prod_j \Delta_{\k_j}^* A_j^{-1}
- \prod_j \Delta_{\k_j}^* {A'}_j^{-1})).
\end{eqnarray*}
Applying Lemma \ref{L:continuity} (3) here we
have $$\ord_p(\fG(B_m)-\fG'(B_m))\geq \min_{j\in\Z/a\Z}(\ord_p(A_j-A'_j))$$
for all $B_m$ in $M_{d\times d}(\cO_E)$.
Notice that
$\pi^m$-coefficients of $\hG_{\gamma,j}$
is the fixed point of the map $\fG$ above.
Applying Lemma \ref{L:continuity}(1) we have
$$\ord_p(\hG_{\gamma,j}-\hG'_{\gamma,j})
\geq \min_{j\in\Z/a\Z}(\ord_p(A_j-A'_j))\geq \min_{j\in\Z/a\Z}(\ord_p(\hA_j-\hA'_j)).$$
This finishes our proof.
\end{proof}

Recall that $\cR_{c,E}$ is the ring of power series
$\sum_{s=0}^{\infty}a_s \pi^s$ in $E[[\pi]]$
with $\ord_p a_s \geq - \frac{s}{c}$.
Let $\chi$ be the cyclotomic character on $\Gamma_K$.
For any $\gamma\in \Gamma_K$ let
$\beta_{\gamma,n}:= \prod_{i=1}^{n}(\chi(\gamma)^i-1)$ for any $n\geq 0$.

\begin{proposition}\label{P:fontaine}
Let $\gamma\in\Gamma_K$ be of infinite order.
For any given diagonal matrix $\hG_{\gamma,j}\in \Id+\pi
M_{d\times d}(\cO_E[[\pi]])$ and $\ord_p A_j = 0$, there is unique
$\hA_j\equiv A_j\bmod \pi$ in $M_{d\times d}(E[[\pi]])$ with
$\pi^n$-coefficients in $\beta_{\gamma,n}^{-1}\cO_E$ such that
$\hG_{\gamma,j-1} \hP_j^\gamma = \hP_j\hG_{\gamma,j}^{\varphi}$
for all $j\in\Z/a\Z$ and for all $n\geq 0$.
In particular, if $A_j \in \GL_d(\cO_E)$
then there exists a geometric generator $\gamma$ of $\Gamma_K$ such that
we may have lifting $\hA_j\in M_{d\times d}(\cR_{(p-1)^2/p,E})$ of $A_j$.
\end{proposition}

\begin{proof}
Write $s_{j,n}:=\sum_{r=0}^{n}S_{j,r}\pi^r$ in $M_{d\times d}(\cO_E[[\pi]])$.
Let $f$ be defined as
$$f(s_j)=\hG_{\gamma,j-1}s_j^\gamma - s_j\hat\Delta_{\k_j}
\hG_{\gamma,j}^\varphi{\hat\Delta_{\k_j}}^{-\gamma}$$
for $j\in\Z/a\Z$.
We claim that for every $n\geq 0$
there exists a unique $S_{j,n}$ in $\beta_{\gamma,n}^{-1}M_{d\times d}(\cO_E)$
such that $f(s_{j,n}) \in \pi^{n+1}\beta_{\gamma,n}^{-1}M_{d\times d}(\cO_E[[\pi]])$.
Notice that for $s_{j,0}=A_j$ we have
$f(s_{j,0})\in\pi M_{d\times d}(\cO_E[[\pi]])$ by hypothesis.
By induction we may write $f(s_{j,n-1}) = - \pi^n R_{j,n-1}$ with
$R_{j,n-1}\in \beta_{\gamma,n-1}^{-1}M_{d\times d}(\cO_E[[\pi]])$.
As $s_{j,n} = s_{j,n-1} + S_{j,n}\pi^n$, we have
$$f(s_{j,n})/\pi^n = \pi^{(\gamma-1)n} \hG_{\gamma,j-1} S_{j,n} - S_{j,n} \hat\Delta_{\k_j}
\hG_{\gamma,j}^\varphi \hat\Delta_{\k_j}^{-\gamma}-R_{j,n-1}.$$
Since $(\chi(\gamma)^n -1 ) S_{j,n} \equiv R_{j,n-1}(0)\bmod \pi$
has a unique solution $S_{j,n}$ in $(\chi(\gamma)^n-1))^{-1}\beta_{\gamma,n-1}^{-1}
M_{d\times d}(\cO_E)$
and hence $S_{j,n}\in \beta_{\gamma,n}^{-1}M_{d\times d}(\cO_E)$.
Since $\hG_{\gamma,j}$ are diagonal and that $q^{\gamma-1} \in 1+\pi\Zp[[\pi]]$,
we have $\hat\Delta_{\k_j} \hG_{\gamma,j}^\varphi \hat\Delta_{\k_j}^{-\gamma}$
lying in $M_{d\times d}(\cO_E[[\pi]])$. Hence
$f(s_{j,n})/\pi^n \in \beta_{\gamma,n}^{-1} M_{d\times d}(\cO_E[[\pi]])$
and therefore $R_{j,n}=f(s_{j,n})/\pi^{n+1}$
lies in $\beta_{\gamma,n}^{-1} M_{d\times d}(\cO_E[[\pi]])$.
This proves our claim and hence we have a unique solution
$\hA_j := \lim_{n\rightarrow\infty} s_{j,n} = \sum_{n=0}^{\infty}S_{j,n}\pi^n$
as the $\pi$-adic limit.

{}From the above, for every $j\in\Z/a\Z$
and $n\geq 0$ we have $\ord_p S_{j,n} \geq - \ord_p\beta_{\gamma,n}$
for any $\gamma\in \Gamma_K$ of infinite order. Note
that $\inf_{\gamma\in\Gamma_K}\ord_p \beta_{\gamma,n} \leq
\lfloor\frac{n}{(p-1)^2/p} \rfloor$ for $p\neq 2$ (see \cite[IV.I]{BE04}).
We may choose a suitable generator $\gamma\in \Gamma_K$
(of infinite order) such that $\ord_p S_{j,n} \geq -\ord_p\beta_{\gamma,n}\geq
-\lfloor\frac{n}{(p-1)^2/p} \rfloor$ for all $n\geq 1$. This proves that
$\hA_j=\sum_{n=0}^{\infty}S_{j,n}\pi^n$
lies in $M_{d\times d}(\cR_{\frac{(p-1)^2}{p},E})$.
By the continuity of the Wach lifting as shown
in Proposition \ref{P:continuity}, our statement follows.
\end{proof}

Below in Theorem \ref{T:continuity}
we should prove a result generalizing \cite[Proposition IV.1.3]{BE04}.
Namely, we will show that if $N(T)$ is an integral embedded Wach module
defined by $(\hP_j,\hG_j)$ with $\hP_j\equiv P_j\bmod \pi$
then for any $P'_j$ in a small $p$-adic neighborhood of $P_j$
we can find an embedded integral Wach module $N(T')$
that is close to $N(T)$.

\begin{lemma}\label{L:wach}
Let $\hG\in \Id + \pi M_{d\times d}(\cO_E[[\pi]])$.
Then there exists a unique $M \in M_{d\times d}(\cR_{(p-1)^2/p,E})\cap \Id+\pi M_{d\times d}(E[[\pi]])$
such that
$\hG M^\gamma = M \hG$.
\end{lemma}
\begin{proof}
This proof is similar to that of Proposition \ref{P:fontaine} above.
Write $f(M) = \hG M^\gamma - M \hG$.
For any $n\geq 0$, write $M_{(n)} = M_0+M_1\pi+\ldots+M_n\pi^n$.
We should show by induction on $n\geq 0$ that
there exists $M_n \in \beta_{\gamma,n}^{-1} M_{d\times d}(\cO_E)$
such that $f(M_{(n)}) \in \pi^{n+1} \beta_{\gamma,n}^{-1} M_{d\times d}(\cO_E[[\pi]])$.
This is clear for $n=0$. If it holds for $n-1$,
then $f(M_{(n-1)})=\pi^n R_{(n-1)}$ for
some $R_{(n-1)} \in \beta_{\gamma,n-1}^{-1} M_{d\times d}(\cO_E[[\pi]])$.
We have
\begin{eqnarray*}
f(M_{(n)})/\pi^n
&=& \beta_{\gamma,n-1}^{-1} R_{(n-1)} + \pi^{n(\gamma-1)}\hG M_n - M_n \hG\\
&\equiv & \beta_{\gamma,n-1}^{-1}R_{(n-1)}(0) + (\chi(\gamma)^n-1)M_n \bmod \pi.
\end{eqnarray*}
There is a unique solution to the last congruence
$M_n= (\chi(\gamma)^n-1)^{-1} \beta_{\gamma,n-1}^{-1} R_{(n-1)}(0)$
lies in $\beta_{\gamma,n}^{-1} M_{d\times d}(\cO_E)$ as we claimed.
On the other hand, we observe that
$f(M_{(n)})/\pi^n \in \beta_{\gamma,n}^{-1} M_{d\times d}(\cO_E[[\pi]])$
and hence $R_{(n)} \in \beta_{\gamma,n}^{-1} M_{d\times d}(\cO_E[[\pi]])$.
This finishes the proof of our induction.
By a similar argument as in Proposition \ref{P:fontaine} we
see that $M \in M_{d\times d}(\cR_{(p-1)^2/p,E})$.
\end{proof}

\begin{theorem}\label{T:continuity}
Write $\epsilon = \lfloor (m - 1)p/(p-1)^2 \rfloor$
where $m=\ink+1$ (or $m=\fk$ if
$k_{j1}\geq \ldots \geq k_{j,d-1}>k_{jd}=0$).
Let $D \in \MF{ad}(E\otimes_\Qp K)$ with
$\Mat(\varphi|_{D_j})=A_j\Delta_{\k_j}$ for some $A_j\in\GL_d(\cO_E)$.
Suppose $D$ can be lifted to an integral embedded Wach module
$N(T)$. Then any $D'\in \MF{ad}(E\otimes_\Qp K)$
with $\Mat(\varphi|_{D'_j}) = A_j'\Delta_{\k_j}$
and with $\ord_p(A'_j - A_j)\geq \epsilon + i$
can be lifted to an integral embedded Wach module $N(T')$ such that
$$ N(T') \equiv N(T) \bmod \m_E^i.$$
\end{theorem}
\begin{proof}
Let $(\hA_j\hat{\Delta}_{\k_j},\hG_{\gamma,j})$ be the matrices of
$(\varphi,\gamma)$ that defines the integral embedded Wach module
$N(T)$.
Since $A'_j\equiv A_j \bmod p^n$ we
may write $A'_j = (\Id+ p^n M_j) A_j$ for some matrix $M_j\in M_{d\times d}(\cO_E)$.
By Lemma \ref{L:wach} above, there exists (a unique)
$\hM_j\equiv M_j$ and $\hM_j \in M_{d\times d}(p^C\cR_{(p-1)^2/p,E})$ such that
$\hG_{\gamma,j-1}\hM_j^\gamma = \hM_j\hG_{\gamma,j-1}$
(where $C=\ord_p M_j$).
Write $\hM_{j,m-1} := (p^\epsilon\hM_j \bmod \pi^m)$,
it lies in $M_{d\times d}(\cO_E[\pi])$ and we have
$$\hG_{\gamma,j-1}\hM_{j,m-1}^\gamma - \hM_{j,m-1} \hG_{\gamma,j-1} \in \pi^m M_{d\times d}(\cO_E[[\pi]]).$$
Therefore,
$$
\hG_{\gamma,j-1}(\Id+p^{n-\epsilon}\hM_{j,m-1})^\gamma -
(\Id+p^{n-\epsilon}\hM_{j,m-1}) \hG_{\gamma,j-1}
\in \pi^m M_{d\times d}(\cO_E[[\pi]]).
$$
Because $\Id+p^{n-\epsilon} \hM_{j,m-1}$ lies in $\GL_d(\cO_E[[\pi]])$
we may multiply its inverse on the right-hand-side of the above equation and get
\begin{eqnarray*}
& & \hG_{\gamma,j-1} - \hA'_j\hat\Delta_{\k_j} \hG_{\gamma,j} (\hA'_j\hat\Delta_{\k_j})^{-\gamma}\\
&=& \hG_{\gamma,j-1} - (\Id+p^{n-\epsilon}\hM_{j,m-1})\hG_{\gamma,j-1}(\Id+p^{n-\epsilon}\hM_{j,m-1})^{-\gamma}\\
&\in & \pi^m M_{d\times d}(\cO_E[[\pi]]).
\end{eqnarray*}
Then we shall apply Proposition \ref{P:unique-split} to find the existence
of integral matrix $\hG'_{\gamma,j}$ that will make
$\hG'_{\gamma,j-1}(\hA'_j\hat\Delta_{\k_j})^\gamma = \hA'_j\hat\Delta_{\k_j}(\hG'_{\gamma,j})^\varphi$.
Hence $(\hA'_j\hat{\Delta}_{\k_j}, \hG'_{\gamma,j})$ defines
an integral embedded Wach module $N(T')$.
If we write $n=\epsilon+i$ with $i>0$.
Then we have $\hA'_j - \hA_j \in p^i M_{d\times d}(\cO_E[[\pi]])$
and hence $\hA'_j \equiv \hA_j \bmod \m_E^i$ in $M_{d\times d}(\cO_E[[\pi]])$.
Therefore $N(T')\equiv N(T)\bmod \m_E^i$.
\end{proof}

Let $|\cdot|_p$ be the metric defined in Remark \ref{R:metric-2}.

\begin{corollary}[Theorem C]\label{C:continuity}
Write $\epsilon = \lfloor (m - 1)p/(p-1)^2 \rfloor$
where $m=\ink+1$ (or $m=\fk$ if
$k_{j1}\geq \ldots \geq k_{j,d-1}>k_{jd}=0$).
Let $V,V'\in\Rep_{\cris/E}^\k(G_K)$ with
Galois stable lattices $T,T'$, respectively.
If $\N(T)$ exists, and if $\dist(V,V') \leq p^{-(i+\epsilon)}$ then
$\N(T')$ exists and $\N(T)\equiv \N(T')\bmod \m_E^i$.
\end{corollary}
\begin{proof}
By Theorem \ref{T:continuity}, it remains to show that $\N(T)\equiv \N(T')\bmod \m_E^i$. Write $A,A'$ in $\GL_d(\cO_E)^a$ for representing $V,V'$ respectively according to Theorem A'.
By our hypothesis and by Remark \ref{R:metric-2},
we have $\ord_p(A_j-A'_j)\geq i+\epsilon$ for every $j$.
By applying Theorem \ref{T:continuity}, we conclude that $\N(T)\equiv \N(T')\bmod \m_E^i$.
\end{proof}

Below we use our result above to show that when $\bar{k}_1<p-1$ we can lift
every strongly divisible lattice to a lattice in an \'etale
$(\varphi,\Gamma)$-module of Fontaine. (When $K=\Qp$ this is known to Fontaine-Laffaille.) In fact, in this case there is an integral Wach lifting map
$\Upsilon: \GL_d(\cO_E)^a/\sim_\k \longrightarrow \GL_d(\cO_E[[\pi]])^a$
(see Remark \ref{R:Berger-Wach}).

\begin{theorem}\label{T:FL}
Let $\bar{k}_1 < p-1$ (and $\bar{k}_1\leq p-1$ if
$k_{j1}\geq \ldots \geq k_{j,d-1} > k_{jd} = 0$ for all $j\in\Z/a\Z$).
Then every embedded strongly divisible lattices in a weakly admissible filtered
$\varphi$-module $D=\dcris^*(V)$ of $\HP_{\Z/a\Z} =\k$
can be lifted to an embedded integral Wach module.
\end{theorem}
\begin{proof}
Choose a (geometric) generator $\gamma\in \Gamma_K$ as in Proposition \ref{P:fontaine}.
Since for $1\leq n\leq p-2$ we have $\chi(\gamma)^n\not \equiv 1\bmod p$,
we have $\ord_p\beta_{\gamma,n} = 0$ (where $\beta_{\gamma,n}$
is as defined in Proposition \ref{P:fontaine}).
So $(\hA_j \bmod \pi^{p-1})$ in Proposition \ref{P:fontaine} lies in $M_{d\times d}(\cO_E[[\pi]])$.
Since $\bar{k}_1 < p-1$  ($\bar{k}_1\leq p-1$ if $k_{j,d-1}> 0$)
we may apply Proposition \ref{P:unique-split} to $(\hA_j\bmod \pi^{p-1})$ above
and $\hG_{\gamma,j,p}=\Id$,  and
obtain the desired integral matrices $\hP_j$ and $\hG_{\gamma,j}$.
Since $\Gamma_K$ is geometrically generated by $\gamma$
and the Wach lifting map $\fW$ is continuous
by Proposition \ref{P:continuity}, these matrices defines
an embedded integral Wach module.
\end{proof}

\begin{theorem}[Theorem B]
\label{T:B}
Let $\bar{k}_1 < p-1$ (and $\bar{k}_1\leq p-1$ if
$k_{j1}\geq \ldots \geq k_{j,d-1} > k_{jd} = 0$ for all $j\in\Z/a\Z$).
For any Galois stable lattice $T$ in $V\in \Rep_{\cris/E}(G_K)$,
$\dcris^*(T)$ is a strongly divisible lattice in $\dcris^*(V)$.
Conversely, every strongly divisible lattice $L$ in $\dcris^*(V)$
is for the form $L=\dcris^*(T)$ for some Galois stable lattice $T$.
\end{theorem}
\begin{proof}
For the first statement we shall show that $L:=\N(T^*)/\pi\N(T^*)$ is strongly
divisible over $\cO_K\otimes_\Zp \cO_E$ in $\dcris^*(V)$.
By our construction above and the hypothesis on weight,
we know that the filtration induced on $L$ from $\N(V^*)$ and
that $\N(T^*)$ are identical.
Split $L\cong \prod_{j\in\Z/a\Z} L_j$, we have
$\Mat(\varphi|_{L_j}) = A_j\Delta_{\k_j}$
and hence $L$ is strongly divisible by Proposition \ref{P:simultaneous}.

Conversely, by Theorem \ref{T:FL} and our hypothesis,
every strongly divisible lattice $L$ of $\dcris^*(V)$
is lifted to an integral Wach module $\N(T^*)$ for some Galois stable lattice
$T^*$ of $V$. Write $T$ for the dual of $T^*$.
Notice that in this lifting, the filtration on $\N(T^*)/\pi\N(T^*)$ coincides with
that induced from $\N(V^*)/\pi\N(V^*)$.
Then by \cite[Theorem III.13]{BE04} and the discussion at the beginning of
Section \ref{S:4.2}, we have $\dcris^*(T)\cong \N(T^*)/\pi\N(T^*)
\cong L$.  This proves Theorem B.
\end{proof}

\section{Mod $p$ reduction of crystalline representation in dimensional $2$ cases}
\label{S:5}

\subsection{Classification of crystalline mod $p$ reductions}
\label{S:5.1}

Let $q=\varphi(\pi)/\pi=((1+\pi)^p-1)/\pi=\pi^{p-1}+\ldots+p$.
For any $V\in\Rep_{\cris/E}(G_K)$ we shall denote the $\varphi$-action
on $D:=\dcris^*(V)=\prod_{j\in\Z/a\Z}D_j$ by $\Mat(\varphi)=(\Mat(\varphi|_{D_0}),\ldots,\Mat(\varphi|_{D_{a-1}}))$.
Let $\vk=(k_0,\ldots,k_{a-1}) \in \Z_{\geq 0}^a$, and let
$\varepsilon_{a,\vk}$ be the crystalline character of $G_K$ whose $\dcris^*(\varepsilon_{a,\vk})$ has $\Mat(\varphi)=(p^{k_0},\ldots,p^{k_{a-1}})$,
same as that given in Remark \ref{R:dim-2}(2).

Let $\varepsilon_a$ be the crystalline character in $\Rep_{\cris/E}^{(1)}(G_K)$ such that $\dcris^*(\varepsilon_a)$ has $\Mat(\varphi)=(p,1,\ldots,1)$. Note that $\varepsilon_1$ is the cyclotomic character of $G_\Qp$ and $\prod_{j\in\Z/a\Z}\varepsilon_a^{\sigma^j} = \varepsilon_1$.
Then we have $\varepsilon_a\equiv \omega_a\bmod p$ where $\omega_a$ is
a fundamental character of $I_K$ (extends to $G_K$).
Notice that $\varepsilon_{a,\vk}=\prod_{j=0}^{a-1}\varepsilon_a^{k_j\sigma^j}$.
We first prepare some simple fact about crystalline characters and their reduction.

\begin{proposition}\label{P:omega_a}
Every crystalline character $V$ of $G_K$ of embedded
Hodge polygon $\HP_{\Z/a\Z}=(k_0,\ldots,k_{a-1})\in \Z_{\geq 0}^a$ is
of the form $\varepsilon_{a,\vk} \otimes \eta$
for an unramified character $\eta$. We have
\begin{eqnarray*}
\varepsilon_{a,\vk}
&\equiv &
\omega_a^{\sum_{j\in\Z/a\Z}k_jp^j}\bmod \m_E.
\end{eqnarray*}
The \'etale $(\varphi,\Gamma)$-module
$\D(\varepsilon_{a,\vk})$ over $\Zp[[\pi]]$ is determined uniquely by
$$\Mat(\varphi)=(q^{k_0},\ldots,q^{k_{a-1}}) \in \Zp[\pi]^a.$$
The \'etale $(\varphi,\Gamma)$-module
$\D(\omega_a^{\sum_{j\in\Z/a\Z}k_j p^j})$
over $\Fp[[\pi]]$ is determined uniquely by:
$$\Mat(\varphi) = (\pi^{(p-1)k_0}, \ldots, \pi^{(p-1)k_{a-1}})\in \Fp[[\pi]]^a.$$
\end{proposition}
\begin{proof}
By Remark \ref{R:dim-2}(2) we know that
$V\cong \varepsilon_{a,\vk}\otimes \eta$
for some unramified character $\eta$.
By Wach theory in Section \ref{S:4.2},
$\dcris^*(\varepsilon_{a,\vk})$ has
$\Mat(\varphi)=(q^{k_0},\ldots,q^{k_{a-1}})$.
Write $\dcris^*(\varepsilon_{a,\vk})=\prod_{j\in\Z/a\Z}D_j$, then
by Lemmas \ref{L:commute} and \ref{L:banach-solution1}(ii)
we found that there are unique solution to $\Mat(\gamma)$ for every
$\gamma \in \Gamma_K$. Take mod $p$ reduction on $\D(\varepsilon_{a,\vk})$ over $\Zp[[\pi]]$,
we obtain the desired $(\varphi,\Gamma)$-module over $\Fp[[\pi]]$,
our assertion follows by Fontaine's theory of $(\varphi,\Gamma)$-modules.
\end{proof}

Classification of 2-dim irreducible crystalline
representation in $\Rep_{\cris/\Qp}(G_K)$ was done by Breuil, see \cite[Lecture 5]{Br00} (see also \cite[Section 3.1]{Br02} and \cite[Section 3.1]{BM02}).
Our goal here is generalizing some of his work to the category $\Rep_{\cris/E}(G_K)$.
We formulate an irreducibility criterion in the following theorem, from which
Theorem D follows immediately (i.e., the equivalence between (i) and (iv)).
Recall the convention that $\sum_{j\in \emptyset}k_j = 0$ for empty set $\emptyset$.

\begin{theorem}\label{T:irreducible}
Let $\vk=(k_0,\ldots,k_{a-1})\in\Z_{\geq 0}^a$.
Let $V\in \Rep_{\cris/E}^\vk (G_K)$ be of dimension $2$, and
let $\Theta^{-1}(V)=A=(A_j)_{j\in\Z/a\Z} \in \GL_2(\cO_E)^a$.
Then the followings are equivalent
\begin{enumerate}
\item[(i)] $V$ is reducible, containing some crystalline character $\tilde{\psi}_1$.
\item[(ii)] There exist
$w_j\in\cO_E^*$ or $p^{k_j}\cO_E^*$, and
$\left(
                                            \begin{array}{c}
                                              x_j \\
                                              y_j \\
                                            \end{array}
                                          \right)\neq \0
$ in $\cO_E$ such that
\begin{eqnarray}\label{E:graph}
A_j\Delta_{k_j}\left(
                  \begin{array}{c}
                    x_j \\
                    y_j \\
                  \end{array}
                \right)
&=&
w_j\left(
   \begin{array}{c}
     x_{j-1} \\
     y_{j-1} \\
   \end{array}
 \right)
\end{eqnarray}
for all $j\in\Z/a\Z$.
The character $\tilde\psi_1$ has embedded $\HP_{\Z/a\Z}(\tilde\psi_1)_j=k_j$ if $w_j\in p^{k_j}\cO_E^*$ and $=0$
if $w_j\in \cO_E^*$.
\item[(iii)]
$P_{(j),A}=\Mat(\varphi^a)$ for one (equivalently all)
$j$ have an eigenvalue whose $p$-adic valuation equal to $\sum_{j\in J}k_j$ for some subset $J$ of $\Z/a\Z$.
\item[(iv)] The $p$-adic Newton polygon of $P_{(j),A}$ (for one and all $j$) has a slope equal to $\sum_{j\in J}k_j$ for some subset $J$ of $\Z/a\Z$.
\item[(v)] Let $t=\tr(P_{(j),A})=\tr(\Mat(\varphi^a))$, then either
$$\ord_p t > \frac{1}{2}\sum_{j\in\Z/a\Z}k_j =
\sum_{j\in J}k_j$$ or
$$\frac{1}{2}\sum_{j\in\Z/a\Z}k_j \geq \ord_p t =\sum_{j\in J}k_j$$ for some
subset $J$ of $\Z/a\Z$.
\end{enumerate}
\end{theorem}
\begin{proof}
The equivalence (i)$\Longleftrightarrow$(ii)
follows from the same equivalence in Proposition \ref{P:irreducibility-2}.
The equivalence (iii)$\Longleftrightarrow$(iv)$\Longleftrightarrow$(v)
follows from basic $p$-adic analysis and theory of $p$-adic Newton polygon, since roots
of the characteristic polynomial of $P_{(j),A}$ are precisely the eigenvalue.
The slopes  of the $p$-adic Newton polygon of $P_{(j),A}$
are precisely the $p$-adic valuation of the eigenvalues.

We show below that (ii)$\Longleftrightarrow$(iii).
Notice that these $P_{(j),A}$ are all conjugate to each other in $M_{2\times 2}(E)$
and hence have the same eigenvalue (and characteristic polynomial), one
may prove the statement for one and hence all $j\in\Z/a\Z$ in (iii).
Suppose (ii) holds. Write $w=w_0w_1\cdots w_{a-1}$,
we have $\ord_pw = \sum_{j\in J}k_j$ for some
subset $J$.
By applying the iteration in (ii), there are nonzero solution
$\left(
   \begin{array}{c}
     x_j \\
     y_j \\
   \end{array}
 \right)
 $
to
\begin{eqnarray}\label{E:1}
P_{(j),A}
\left(
   \begin{array}{c}
     x_j \\
     y_j \\
   \end{array}
 \right)
 &=&
 w\left(
                  \begin{array}{c}
                    x_j \\
                    y_j \\
                  \end{array}
                \right)
\end{eqnarray}
which shows that (iii) holds.
Conversely, suppose (iii) holds. That is
$P_{(j),A}$ has an eigenvalue $w$ with $p$-adic valuation
equal to $\sum_J k_j$ for some subset $J$ of $\Z/a\Z$.
Then we have a nonzero solution
for (\ref{E:1}) for one $j$. Let the rest of the
vectors $
\left(
  \begin{array}{c}
    x_j \\
    y_j \\
  \end{array}
\right)
$
over $E$ be computed via the iteration in (\ref{E:graph}), and
all can be chosen over $\cO_E$ without problem.
This proves (ii).
\end{proof}

\begin{remark}

Our irreducibility criterion above in $\Rep_{\cris/E}(G_K)$
is very similar to the existing one for $\Rep_{\cris/\Qp}(G_K)$,
compare this with Remark \ref{R:irreducible}. For $a=1$ we recover that $V$ is irreducible if and only if Newton polygon has only positive slopes.

By Theorem \ref{T:irreducible}, it is easy to see that, if $V$ is reducible then
$$ V \cong \left(
  \begin{array}{cc}
    \prod_{j\in J}\varepsilon_a^{k_j\sigma^j} & \star \\
    0 & \prod_{j\not\in J}\varepsilon_a^{k_j\sigma^j} \\
  \end{array}
\right)\otimes \eta
$$
for some unramified character $\eta$ and some subset $J$ of $\Z/a\Z$.
Compare this to Remark \ref{R:dim-2}(4).
\end{remark}

Let $V$ be a 2-dimensional $E$-linear crystalline representation
in $\Rep_{\cris/E}^\vk (G_K)$ with $\vk=(k_0,\ldots,k_{a-1})\in\Z_{\geq 0}^a$.
Let $\rho:=\bar{V}$ be the semisimplification of reduction mod $p$ of $V$.
The following theorem generalizes Breuil's classification
\cite[Proposition 6.1.2]{Br02} for the case $K=\Qp$, and proves Theorem E.

\begin{theorem}[Theorem E]\label{T:E}
(i) Let $0\leq k_j\leq p-1$ for all $j$. If $\rho$ is reducible, then
$$
\rho|_{I_K}
\cong
\left(
  \begin{array}{cc}
    \omega_a^{\sum_{j\in \Z/a\Z}k_jp^j} & \star \\
    0 & 1 \\
  \end{array}
\right)\otimes \bar\eta
$$
for some character $\bar\eta$ that extends to $G_K$.

(ii) Let $0\leq k_j\leq p-1$ for all $j$ and $\vk\neq 0$.
If $\rho$ is irreducible, then
$$
\rho\cong \ind(\omega_{2a}^{\sum_{j\in\Z/a\Z}k_jp^j})\otimes\bar\eta$$
for some character $\bar\eta$; i.e.,
$$
\rho|_{I_K}\cong
\left(
           \begin{array}{cc}
           \omega_{2a}^{\sum_{j\in \Z/a\Z}k_jp^j } & 0 \\
           0     & \omega_{2a}^{p^a \sum_{j\in \Z/a\Z}k_jp^j} \\
        \end{array}
      \right) \otimes \bar\eta.
$$
for some character $\bar\eta$ that extends to $G_K$.
If $\vk=\0$ then $\rho\cong \ind(\omega_{2a}^{p^a-1})\otimes\bar\eta$ for some
character.
\end{theorem}
\begin{proof}
(i)
Since $\rho$ is reducible,
by Breuil's \cite[Corollary 2.9]{Br07} we have that
$$
\rho|_{I_K}\cong
\left(
  \begin{array}{cc}
    \omega_a^{\sum_{j\in \Z/a\Z}r_jp^j} & \star \\
    0 & 1 \\
  \end{array}
\right)\otimes \bar\eta
$$
for some character $\bar\eta$ that extends to $G_K$
and some $0\leq r_j\leq p-1$ and not all equal to $p-1$.
Thus we have upon a twist by a character
\begin{eqnarray}\label{E:compare2}
\det(\rho)\cong \omega_a^{\sum_{j\in\Z/a\Z}r_j p^j}.
\end{eqnarray}

On the other hand, let $T$ be a Galois stable lattice in $V$.
Write $(L=\prod_{j\in\Z/a\Z}L_j,\varphi):=\dcris^*(T)$
for a strongly divisible lattice in $D=\dcris^*(V)$. With respect to basis adapted to filtration, we have $\Mat(\varphi|_L)=(A_0\Delta_0,\ldots,A_{a-1}\Delta_{a-1})$
where $A_j\in \GL_2(\cO_E)$ and $\Delta_j=\diag(p^{k_j},1)$.
By Theorem \ref{T:FL}, we may lift $T$ to integral Wach module $N$
with $\Mat(\varphi|_N)=(\hA_0\hat\Delta_0,\ldots,\hA_{a-1}\hat\Delta_{a-1})$
where $\hA_j\equiv A_j\bmod \pi$ lies in $\GL_2(\cO_E[[\pi]])$ and
$\hat\Delta_j = \diag(q^{k_j},1)$ where $q=\varphi(\pi)/\pi$.
Notice that $q\equiv \pi^{p-1}\bmod p$, we have
the \'etale $(\varphi,\Gamma)$-module
$\N(\bar{V}):= \N(T/\m_E T) \cong N/\m_E N$ over $\bar\F_p$
given by
$$\Mat(\varphi|_{\bar{N}})=
\left(
\bar\hA_0
\left(
            \begin{array}{cc}
              \pi^{(p-1)k_0} & 0 \\
              0 & 1 \\
            \end{array}
\right),
          \ldots,
\bar\hA_{a-1}
\left(
  \begin{array}{cc}
    \pi^{(p-1)k_{a-1}} & 0 \\
    0 & 1 \\
  \end{array}
\right)
\right)
$$
that commutes with $\Gamma_K$-action.
Write $u_j:=\det(\bar\hA_j)=(\det(\hA_j)\bmod \m_E)$ for every $j$, it is clear that $u_j\in \bar\F_p[[\pi]]^*$. Thus $\D(\det(\rho))$ has
$$\Mat(\varphi)=\det(\Mat(\varphi|_{\bar{N}})) =(u_0\pi^{(p-1)k_0},\ldots,u_{a-1}\pi^{(p-1)k_{a-1}}).$$
Recall from Proposition \ref{P:omega_a} that
$\D(\omega_a^{\sum_{j\in\Z/a\Z}k_jp^j})$ has $\Mat(\varphi)=(\pi^{(p-1)k_0},\ldots,\pi^{(p-1)k_{a-1}})$,
we notice that upon an unramified twist
$\det(\rho)\cong\omega_a^{\sum_{j=0}^{a-1}k_jp^j}.$
Comparing this with (\ref{E:compare2}),
since we have $\omega_a^h=1$ if and only if
$(p^a-1)|h$, we have $k_j=r_i$ if $\vk\neq (p-1,\ldots,p-1)$. If $\vk=(p-1,\ldots,p-1)$
we have $\det(\rho)\cong \Id$ and hence
our assertion is not affected if we set $r_j=p-1=k_j$ for every $j$. This finishes our proof of part (i).

(ii)
This part follows a similar argument as part (i).
Since $\rho$ is irreducible,
by Breuil's \cite[Corollary 2.9]{Br07} we have that
$$
\rho|_{I_K}\cong
\left(
  \begin{array}{cc}
    \omega_{2a}^{\sum_{j\in \Z/a\Z}r_jp^j} & 0 \\
    0 & \omega_{2a}^{p^a\sum_{j\in\Z/a\Z}r_jp^j} \\
  \end{array}
\right)\otimes \bar\eta
$$
for some character $\bar\eta$ that extends to $G_K$
where $1\leq r_0\leq p$,  $0\leq r_j\leq p-1$ for $1\leq j\leq a-1$,
and $\vec{r}\neq (p,p-1,\ldots,p-1)$. For the case $r_0=p$
consider the $p$-adic round-up,  we may assume from now on that
$0\leq r_j\leq p-1$ for all $j$ and $r_0\neq 0$.
Upon a twist we have
\begin{eqnarray}\label{E:compare1}
\det(\rho)\cong \omega_a^{\sum_{j\in\Z/a\Z}r_j p^j}.
\end{eqnarray}
The rest of the argument in part (i) carries over here without any problem since the weight $\vk$ satisfies the same hypothesis of Theorem \ref{T:FL},
hence we may Wach lift $V$, then reduce mod $P$ and obtain that
$\det(\rho)\cong \omega_a^{\sum_{j\in\Z/a\Z}k_jp^j}\otimes \bar\eta$ for some
character $\bar\eta$. Now we compare this with (\ref{E:compare1}):
If $\vk\neq \0$, then we have that $k_j=r_j$ for all $j$.
If $\vk=\0$, then we have $r_j=p-1$ for all $j$. In this case
$\rho\cong \ind(\omega_{2a}^{p^a-1})\otimes \bar\eta$.
\end{proof}

\subsection{Some analytic families in $\Rep_{\cris/E}^\vk (G_K)$}
\label{S:5.2}

The purpose of this subsection is to construct some explicit
families of crystalline representations of dimension $2$.
This allows us to have some crystalline deformation of
prescribed reduction type. We are able to address on the boundary
weight $\vk=(p,\ldots,p)$ in these explicit constructions.
Let $\cO_\Cp$ be the ring of integers of $\Cp$, and
let $\tilde{\E}^+ = \lim_{x\mapsto x^p}\cO_{\Cp}$ be the
ring defined by Fontaine \cite{Fo88a}, and let $\tilde\E=\tilde{\E}^+[1/\pi]$.

We write $\Rep_{\cris/E}^{\vk}(\cdot)$
for the category of $E$-linear crystalline representations of $G_K$
with embedded Hodge polygon $\vk=(k_0,\ldots,k_{a-1})$. Let $\bar{k}_1:=\sum_{j\in\Z/a\Z}k_j/a$.
We define a 2-dimensional crystalline representations
$V_{\vk,\vv,\vu}$ in $\Rep_{\cris/E}^{\vk}(G_K)$.
Write $\Gal(K/\Qp)=\Z/a\Z=\cA\cup \cB$ as a partition.
Let $\vv=(v_0,\ldots,v_{a-1}) \in\cO_E^a$ and let
$\vu=(u_j)_{j\in\cB}$ with $u_j\in\cO_E^*$.
Let $A:=(A_j)_{j\in\Z/a\Z}$ where
$$
A_j=
\left\{
\begin{array}{ll}
\left(\begin{array}{cc}
      0 & -1 \\
      1 & v_j \\
      \end{array}
\right) & \mbox{if $j\in \cA$ or }\\
\left(
  \begin{array}{cc}
    1        & 0 \\
    v_j & u_j \\
  \end{array}
\right) &\mbox{if $j\in \cB$.}
\end{array}
\right.
$$
Recall the bijection $\Theta: \GL_2(\cO_E)^a/\sim_\vk \longrightarrow \Rep_{\cris/E}^{\vk}(G_K)$ per Theorem \ref{T:A-analog}.
Let $V_{\vk,\vv,\vu}:=\Theta(A)$ in $\Rep_{\cris/E}^{\vk}(G_K)$.

We have a remark on the parameters $\vv$ and $\vu$ in $V_{\vk,\vv,\vu}$.
If we define an equivalence $\sim_\vk$ on $\cO_E^a$,
denoted  by $\vv\sim_\vk \vv'\in \cO_E^a$,
by $v_j=w^{(-1)^j}v_j'$ for some $w\in\cO_E^*$ when $a$ is even; or $\vv=\vv'$ when $a$ is odd, then we have the following property:

\begin{proposition}\label{P:isom=2}
Let $\Z/a\Z=\cA$ and let $k_j>0$ for every $j$.
Then $V_{\vk,\vv} \cong V_{\vk,\vv'}$ if and only if
$\vv\sim_\vk \vv'$.
\end{proposition}
\begin{proof}
By Theorem \ref{T:A-analog}, we have $V_{\vk,\vv}\cong
V_{\vk,\vv'}$ if and only if $A\sim_\vk A'$ where
$A'=(\left(\begin{array}{cc}
      0 & -1 \\
      1 & v'_j \\
      \end{array}
\right))_j$.
Namely, there is $C_j =\left(%
\begin{array}{cc}
  c_{j1} &  c_{j3} \\
  0 & c_{j2} \\
\end{array}%
\right)
\in\GL_2(\cO_E)$
such that $C_{j-1} A_j = A'_j C_j^\flat$
for every $j\in\Z/a\Z$.
This implies that $C_j=\diag(c_{j1},c_{j2})$
and $v_j=\frac{c_{j2}}{c_{j-1,2}}v_j'$, where
$c_{j2}=c_{j-1,1}$ and $c_{j1}=c_{j-1,2}$.
Write $w=c_{02}/c_{a-1,2}$, then
we get $v_j=w^{(-1)^j}v_j'$ for every $j$.
When $a$ is odd this forces $w=1$ and hence
our proposition follows.
\end{proof}

Before our construction of Wach modules, we
retain notations from Lemma \ref{L:solution1}.
Let $b=a$ if $\#\cA$ is even or $b=2a$ if $\#\cA$ is odd.
If $k_j=0$ let $\delta_j=0$; if $k_j>0$ let
\begin{eqnarray}\label{E:delta_j}
\delta_j &:=& - \ord_p(\lambda_b^{(g_{j-1}-f_{j-1})(\varphi)}\bmod \pi^{\fk}).
\end{eqnarray}
We shall need the following lemma for the proof of Theorem \ref{T:family}.

\begin{lemma}\label{L:delta}
Let $\vk=(k_0,\ldots,k_{a-1})\in\Z_{\geq 0}^a$ and let $\bar{k}_1=\sum_{j\in\Z/a\Z}k_j/a$.
For $j\in\Z/a\Z$ then we have $0\leq \delta_j\leq \frac{\fk-1}{p-1}$.
If $\vk=(p,\ldots,p)$ then $\delta_j=0$ for every $j\in\Z/a\Z$.
\end{lemma}
\begin{proof}
By Lemma \ref{L:banach-solution1},
we know $\lambda_b\in\cR_{p-1,\Qp}$ for any $b$ and hence
by the very definition we have
$0\leq \delta_j\leq \frac{\fk-1}{p-1}$.

It remains to show that
for $k_0=\ldots =k_{a-1}= p$, we have $\delta_j=0$ for all $j$.
By Lemma \ref{L:solution1} and its proof,
$f_j(\varphi),g_j(\varphi)\in p\Z[\varphi]$ for all $j$ since
they are algebraic combinations of $k_0,\ldots,k_{a-1}$. Thus
$\lambda_b^{(g_{j-1}-f_{j-1})(\varphi)}\in\lambda_b^{p\Z[\varphi]}$
and hence there exists
$C:=c_0+c_1\pi+\ldots+c_{p-2}\pi^{p-2}+p^{-1}c_{p-1}\pi^{p-1}$ in
$\cR_{p-1,\Qp}$ with $c_i\in\Zp$ such that
$(\lambda_b^{(g_{j-1}-f_{j-1})(\varphi)}\bmod \pi^p) =(C^p \bmod
\pi^p)$, which is clearly in $\Zp[\pi]$
by examining its binomial expansion. This shows $\delta_j=
-\ord_p(\lambda_b^{(g_{j-1}-f_{j-1})(\varphi)}\bmod \pi^p)=0$ by definition.
\end{proof}

Let
\begin{eqnarray}\label{E:z_j}
z_j &:=&
(p^{\delta_j}\lambda_b^{(g_{j-1}-f_{j-1})(\varphi)}\bmod
\pi^{\fk}),
\end{eqnarray}
notice that $z_j\in\Zp[\pi]$ by the very definition of $\delta_j$ above.

Let $u_j\in\cO_E^*$. Write $X_j$ for a variable on which $\gamma$
and $\varphi$ act trivially, we define for $k_j>0$
\begin{eqnarray}\label{E:P}
\hP_j(X_j) &:=&
\left\{
\begin{array}{ll}
\left(
              \begin{array}{cc}
                0 & -1 \\
                q^{k_j} & z_jX_j\\
              \end{array}
            \right), &\mbox{if $j\in\cA$}\\
\nonumber
\left(
\begin{array}{cc}
q^{k_j}     & 0 \\
q^{k_j}z_jX_j & u_j \\
\end{array}
\right), &\mbox{if $j\in\cB$};
\end{array}
\right.\\
\hG_{\gamma,j} &:= &\left(
          \begin{array}{cc}
            s_j & 0 \\
            0  & t_j \\
          \end{array}
        \right)
\end{eqnarray}
with $s_j,t_j\in 1+\pi \Zp[[\pi]],z_j\in \Zp[[\pi]]$
be as given in Lemma \ref{L:solution1} and above.
If $k_j=0$ then set $\hP_j(X_j):=\hP_j(0)$ and the same $\hG_{\gamma,j}$.

\begin{proposition}\label{P:3.1.2-analog}
Write $m=\fk$.
For any $\gamma \in \Gamma_K$ we have for every $j$
$$
\hG_{\gamma,j-1} - \hP_j(X_j)
\hG_{\gamma,j}^\varphi \hP_j(X_j)^{-\gamma}
\in
\pi^{m}M_{2\times 2}(\Zp[[\pi,X_j]]).$$
Moreover, the left-hand-side vanishes for $k_j=0$.
\end{proposition}

\begin{proof}
Let $j\in\cA$.
A straightforward computation
shows that
$$\hG_{\gamma,j-1}\hP_j(X_j)^{\gamma}-\hP_j(X_j)\hG_{\gamma,j}^{\varphi}
= \left(
    \begin{array}{cc}
      0 & t_j^\varphi - s_{j-1} \\
      q^{k_j}(t_{j-1}q^{(\gamma-1)k_j}-s_j^\varphi)
      & (t_{j-1}z_j^\gamma-t_j^\varphi z_j)X_j \\
    \end{array}
  \right).
$$
Note that the two off-diagonal entries vanish if and only if
(\ref{E:s_j}) holds, which has a unique solution for $(s_j,t_j)_j$.
The diagonal entry vanishes if and only if (\ref{E:h_j}) holds by the
very definition of $z_j$ in (\ref{E:z_j}),
whose solution set is $\cS$ as defined in Lemma \ref{L:solution1}.
We have
$$\hG_{\gamma,j-1}\hP_j(X_j)^{\gamma}-\hP_j(X_j)\hG_{\gamma,j}^{\varphi}
=\left(
  \begin{array}{cc}
    0 & 0 \\
    0 & \pi^{m}\star\\
  \end{array}
\right)
\in
\pi^{m}M_{2\times 2}(\Zp[[\pi,X_j]]).$$
Following an identical argument as that in the proof of
\cite[Proposition 3.1.3]{BLZ04} we have
\begin{eqnarray*}
\hG_{\gamma,j-1} - \hP_j(X_j)
\hG_{\gamma,j}^\varphi \hP_j(X_j)^{-\gamma}
&=&
    \left(
    \begin{array}{cc}
    0 & 0 \\
    0 & \pi^{m}\star\\
    \end{array}
    \right) \hP_j(X_j)^{-\gamma}\\
&=& \left(
    \begin{array}{cc}
    0 & 0 \\
    0 & \pi^{m}\star\\
    \end{array}
    \right)\left(
                        \begin{array}{cc}
                          z_jX_jq^{-k_j} & q^{-k_j} \\
                          -1 & 0 \\
                        \end{array}
                      \right)^\gamma \\
&=&
\left(
    \begin{array}{cc}
    0 & 0 \\
    0 & \pi^{m}\star\\
    \end{array}
    \right)
\in \pi^{m}M_{2\times 2}(\Zp[[\pi,X_j]]).
\end{eqnarray*}
Now suppose $j\in\cB$. We have
$$\hG_{\gamma,j-1}\hP_j(X_j)^{\gamma}-\hP_j(X_j)\hG_{\gamma,j}^{\varphi}
= \left(
    \begin{array}{cc}
      q^{k_j}(q^{(\gamma-1)k_j}s_{j-1}- s_j^\varphi)  &  0 \\
      q^{k_j}z_j(q^{(\gamma-1)k_j}t_{j-1}z_j^{\gamma-1}-s_j^\varphi)X_j
          & (t_{j-1}-t_j^\varphi)u_j \\
          \end{array}
  \right).
$$
Then
\begin{eqnarray*}
\hG_{\gamma,j-1} - \hP_j(X_j)
\hG_{\gamma,j}^\varphi \hP_j(X_j)^{-\gamma}
&=&
\left(
  \begin{array}{cc}
    0 & 0 \\
    \pi^m\star & 0\\
  \end{array}
\right) \in \pi^m M_{2\times 2}(\Zp[[\pi,X_j]]).
\end{eqnarray*}
The case when $k_j=0$ is an easy computation.
\end{proof}

\begin{proposition}\label{P:continuity2}
Let $\hP_j(X_j)$ in $M_{2\times 2}(\Zp[\pi,X_j])$ be as in (\ref{E:P})
for all $j\in \Z/a\Z$
and write $\vX =(X_0,\ldots,X_{a-1})$.
There exist unique matrices $\hG_{\gamma,j}(\vX) \in \Id +
 \pi M_{2\times 2}(\Zp[[\pi,\vX]])$ such that
$$\hG_{\gamma,j-1}(\vX) \hP_{j}(X_{j})^\gamma
=\hP_{j}(X_{j}) \hG_{\gamma,j}(\vX)^\varphi$$
for all $j\in \Z/a\Z$.
\end{proposition}
\begin{proof}
By Proposition \ref{P:3.1.2-analog}, if $k_j=0$
then $\hG_{\gamma,j}(\vX) := \hG_{\gamma,j}$ will do the job.
So we may assume $k_j>0$ for the rest of the proof.
Write $f_{j,\ell} = \hG_{\gamma,j-1,\ell}-
\hP_j(\vX)\hG_{\gamma,j,\ell}^{\varphi}\hP_{j}(\vX)^{-\gamma}$
in $\Zp[[\pi,\vX]]/\pi^{\ell+1}$.
By Proposition \ref{P:3.1.2-analog}, we have
$\hP_j(\vX)\in M_{2\times 2}(\Zp[[\pi,\vX]])$ and
$\hG_{\gamma,j}$ in $M_{2\times 2}(\Zp[[\pi]])$
such that $f_{j,m-1}\in \pi^m M_{d\times d}(\Zp[[\pi,\vX]])$ for all $j$.
By applying a general version of Proposition \ref{P:unique-split} as
explained in Remark \ref{R:addX}, we obtain a unique
$\pi$-adic lifting $\hG_{\gamma,j} \in M_{2\times 2}(\Zp[[\pi,\vX]])$
such that $\hG_{\gamma,j} \equiv \hG_{\gamma,j,m-1}\bmod \pi^m$.
\end{proof}

We have the following lemma generalizing
\cite[Proposition 3.2.1]{BLZ04}.

\begin{lemma}\label{L:family-1}
For $\gamma,\eta\in \Gamma_K$  and for $\valpha=(\alpha_0,\ldots,\alpha_{a-1})
\in \cO_E^a$,
we have
$\hG_{\gamma\eta,j}(\valpha)
=\hG_{\gamma,j}(\valpha)\hG_{\eta,j}(\valpha)^\gamma$.
Moreover, for every $j\in\Z/a\Z$ we have
$$
\hG_{\gamma,j-1}(\valpha)\cdot \hP_j(\alpha_j)^\gamma
=\hP_j(\alpha_j)\cdot\hG_{\gamma,j}(\valpha)^\varphi
$$
so that one can use the matrices
$\hP_j(\alpha_j)$ and $\hG_{\gamma,j}(\valpha)$
to define an integral  embedded Wach
module $N(\valpha):=\prod_{j\in\Z/a\Z} N(\valpha)_j$
over $\cO_K[[\pi]]\otimes_\Zp \cO_E$.
\end{lemma}
\begin{proof}
We know from Proposition \ref{P:continuity2} that
$$\hG_{\gamma,j-1}(\vX) \hP_j(X_j)^\gamma
=\hP_j(X_j) \hG_{\gamma,j}(\vX)^\varphi$$
and if $\gamma,\eta\in\Gamma_K$ then
$\hG_{\gamma\eta,j}(\vX)$ and
$\hG'_{\gamma\eta,j}(\vX) = \hG_{\gamma,j}(\vX)\cdot \hG_{\eta,j}(\vX)^\gamma$
both satisfy the conditions of Proposition \ref{P:continuity2} and hence they are
equal. We then define
$N(\valpha) =\prod_{j\in\Z/a\Z} N(\valpha)_j$ where $N(\valpha)_j$ is
2-dimensional free $\cO_E[[\pi]]$-module which we write
$N(\valpha)_j = \cO_E[[\pi]] n_{j,1} + \cO_E[[\pi]] n_{j,2}$
with basis $n_{j,1}, n_{j,2}$.
We endow it with action $\gamma: N(\valpha)_j\rightarrow N(\valpha)_j$
and $\varphi_j: N(\valpha)_j \rightarrow N(\valpha)_{j-1}$ given by
matrices $\hG_{\gamma,j}(\valpha)$ and $\hP_j(\alpha_j)$ respectively.
\end{proof}

Recall $\delta_j$ defined in (\ref{E:delta_j}).
Let $\vv=(v_0,\ldots,v_{a-1})$ with $v_j\in E$ and $\ord_pv_j \geq \delta_j$.
Let $\valpha:=(p^{-\delta_j}v_j)_j$  (in $\cO_E^a$).
Let $V_{\vk,\vv,\vu}$ be the $E$-linear crystalline representation
such that $E\otimes_{\cO_E} N(\valpha) = \N (V_{\vk,\vv,\vu}^*)$
and let $T_{\vk,\vv,\vu}$ be the $\cO_E$-lattice in
$V_{\vk,\vv,\vu}$ such that $N(\valpha) = \N (T_{\vk,\vv,\vu}^*)$.
The following result generalizes result of \cite{BLZ04}, in fact,
for $a=1$ we recovers \cite[Proposition 3.2.4]{BLZ04} with our
$V_{\vk,\vv,\vu}$ equal to $V_{k,a_p}$ in notation of \cite{BLZ04}.

\begin{proposition}\label{P:cris}
For $\valpha=(\alpha_0,\ldots,\alpha_{a-1})\in \cO_E^a$,
the filtered $\varphi$-module $E\otimes_{\cO_E}N(\valpha)/\pi N(\valpha)$
is isomorphic to a weakly admissible
filtered $\varphi$-module $D_{\vk,\vv,\vu}$ so that
we have $\dcris^*(V_{\vk,\vv,\vu}) = D_{\vk,\vv,\vu}$.
\end{proposition}
\begin{proof}
By definition we have
$$z_j \alpha_j = (p^{-\delta_j}z_j)v_j
=(\lambda_b^{(g_{j-1}-f_{j-1})(\varphi)}
\bmod \pi^{\fk})v_j \equiv v_j\bmod \pi$$
since $\lambda_b\equiv 1\bmod \pi$. Then we find
$$\hP_j(\alpha_j)
=
\left\{
\begin{array}{ll}
\left(
                      \begin{array}{cc}
                        0 & -1 \\
                        q^{k_j} & z_j\alpha_j \\
                      \end{array}
                    \right)
\equiv \left(%
\begin{array}{cc}
  0 & -1 \\
  p^{k_j} & v_j \\
\end{array}%
\right) \bmod \pi
&\mbox{for $j\in\cA$,  and}\\
\left(
\begin{array}{cc}
q^{k_j} & 0\\
q^{k_j}z_j\alpha_j & u_j\\
\end{array}
\right) \equiv
\left(%
\begin{array}{cc}
  p^{k_j} & 0 \\
  p^{k_j}v_j & u_j \\
\end{array}%
\right)
\bmod \pi
&\mbox{for $j\in\cB$.}
\end{array}
\right.
$$
We have $G_{\gamma,j}\equiv \Id\bmod \pi$. Follow a similar
argument as that in \cite[Prop. 3.2.4]{BLZ04}, we find that
$E\otimes_{\cO_E}N(\valpha)/\pi N(\valpha)$ is isomorphic to $D_{\vk,\vv,\vu}$.
This concludes our proof.
\end{proof}

We consider the mod-$p$ reduction behavior of $V_{\vk,\vv,\vu}$.
By mod-$p$ reduction we mean $T_{\vk,\vv,\vu}/\m_E$
for any Galois stable lattice $T_{\vk,\vv,\vu}$ in $V_{\vk,\vv,\vu}$.
Note that the semisimplification of this reduction, denoted by $\bar{V}_{\vk,\vv,\vu}$,
is independent of the choice of lattice.

\begin{theorem}\label{T:family}
(i) For $t=1,2$ let $\vv^{(t)}$ be
such that $\ord_pv^{(t)}_j > \delta_j$ where $\delta_j$ is defined in (\ref{E:delta_j}) for every $j$.
Let $V_{\vk,\vv^{(t)},\vu}$
be the crystalline $E$-linear representation such that
$E\otimes_{\cO_E}N((p^{-\delta_j}v_j^{(t)})_j)=\N(V^*_{\vk,\vv^{(t)},\vu})$, and
let $T_{\vk,\vv^{(t)},\vu}$ be a Galois stable lattice in
$V_{\vk,\vv^{(t)},\vu}$ such that
$N((p^{-\delta_j}v_j^{(t)})_j)=\N(T_{\vk,\vv^{(t)},\vu})$. Suppose
$\ord_p(v_j^{(1)}-v_j^{(2)})\geq \delta_j+i$, then
$$
T_{\vk,\vv^{(1)},\vu} \equiv T_{\vk,\vv^{(2)},\vu} \bmod \m_E^i.
$$

(ii) For any $\vv$ with $\ord_p \vv > \lfloor \frac{\fk-1}{p-1}\rfloor$, or
with $\ord_p \vv > 0$ and $\vk=(p,\ldots,p)$,
we have $\bar{V}_{\vk,\vv,\vu} \cong \bar{V}_{\vk,\0,\vu}$.
\end{theorem}
\begin{proof}
(i) Write $\alpha_j = p^{-\delta_j}v_j$.
By hypothesis,
$\hP_j(\valpha^{(1)})\equiv \hP_j(\valpha^{(2)})$
and
$\hG_{\gamma,j}(\valpha^{(1)})\equiv \hG_{\gamma,j}(\valpha^{(2)})$
modula $\m_E^i$.
Thus $N(\valpha^{(1)})\equiv N(\valpha^{(2)})\bmod \m_E^i$.
Then we have
$T_{\vk,\vv^{(1)},\vu}\equiv T_{\vk,\vv^{(2)},\vu}\bmod \m_E^i$.

(ii) Observe that $\delta_j\leq \frac{\fk-1}{p-1}$,
so $\ord_p \vv> \lfloor\frac{\fk-1}{p-1} \rfloor$
implies that $\ord_p v_j > \delta_j$ for every $j$.
Hence we may apply part (i) to $\vv^{(2)}=\0$.
When $\vk=(p,\ldots,p)$, then $\delta_j=0$ for all $j$ by
Lemma \ref{L:delta}, hence our assertion follows from the same argument
as above.
\end{proof}

Below we shall construct some crystalline deformation of 2-dimensional reducible and irreducible representations in $\Rep_{/\bar{\F}_p}(G_K)$, whose complete classification can be found in Breuil's notes \cite{Br07}.

\begin{proposition}\label{P:reducible-deform}
Let $\vk=(k_0,\ldots,k_{a-1})$
with $0\leq k_j\leq p-1$, or let $\vk=(p,\ldots,p)$.
If $V_{\vk,\vv,\vu}=\Theta((A_j)_j)$ with $j\in\cB$ for all $j$
and $\ord_p\vv > 0$, then we have
$$
\bar{V}_{\vk,\vv,\vec{1}}
\cong
\bar{V}_{\vk,\0,\vec{1}}
\cong
\left(
  \begin{array}{cc}
  \omega_a^{\sum_{j\in\Z/a\Z} k_jp^j}    & 0 \\
    0                                    & 1 \\
  \end{array}
\right) \otimes \bar\eta
$$
for some unramified character $\bar\eta$.
\end{proposition}

\begin{proof}
Let $\bar{N}=\prod_{j\in\Z/a\Z}\bar{N}_j$  denote the $(\varphi,\Gamma)$-module
of $\bar{V}_{\vk,\0,\1}$.
Then
$$\Mat(\varphi|_{\bar{N}}) = \left( \left(
                                 \begin{array}{cc}
                                   \pi^{(p-1)k_j} & 0 \\
                                   0              & 1 \\
                                 \end{array}
                               \right)_j\right)
                               $$
with respect to certain lifted basis adapted to filtration $\bar{N}=\pscal{\vec{e}_1,\vec{e}_2}$.
Thus
\begin{eqnarray*}
\Mat(\varphi^a|_{\bar{N}_{a-1}}) &=& \left(
                                 \begin{array}{cc}
                                   \pi^{(p-1)\sum_{j\in\Z/a\Z}k_jp^j} & 0 \\
                                   0              & 1 \\
                                 \end{array}
                               \right).
\end{eqnarray*}
This says that
$\varphi^a|_{\vec{e}_1}=\pi^{(p-1)\sum_{j\in\Z/a\Z}k_jp^j}$
and $\varphi^a|_{\vec{e}_2}=1$.
Under our hypothesis on embedded Hodge polygon $\vk$,
we find that $\bar{k}_1<p-1$ ($\bar{k}_1\leq p-1$ when
$k_j>0$ for all $j$) and hence
we may apply Fontaine-Laffaille type theorem in Theorem \ref{T:FL}.
The sub-object $\pscal{\vec{e}_1}=\N(\omega_a^{\sum_{j\in\Z/a\Z}k_jp^j})$
by Proposition \ref{P:omega_a}. Thus
$
\bar{V}_{\vk,\0,\vec{1}} \cong
\omega_a^{\sum_{j\in\Z/a\Z}k_jp^j}\oplus \Id.
$

Now by Theorem \ref{T:family}(ii),
for any $\ord_p(\vv)> \lfloor\frac{\fk-1}{p-1} \rfloor=0$
we have $\bar{V}_{\vk,\vv,\vec{1}}\cong \bar{V}_{\vk,\0,\vec{1}}$.
Combining these above, we have
$$
\bar{V}_{\vk,\vv,\vec{1}}
\cong
\bar{V}_{\vk,\0,\vec{1}}
\cong
\left(
  \begin{array}{cc}
  \omega_a^{\sum_{j\in\Z/a\Z}k_jp^j}    & 0 \\
    0                                    & 1 \\
  \end{array}
\right) \otimes \bar\eta
$$
for some unramified character $\bar\eta$.
\end{proof}

\begin{proposition}
\label{P:irreducible-deform}
Let $\vk=(k_0,\ldots,k_{a-1})$ such that $0\leq k_j\leq p-1$ for every $0\leq j\leq a-1$ or $\vk=(p,\ldots,p)$.
If $V_{\vk,\vv,\vu}=\Theta((A_j)_j)$ with $j\in \cB$ for all $0\leq j\leq a-2$
and $a-1\in \cA$ with $\ord_p(\vv)> 0$ then we have
\begin{eqnarray*}
\bar{V}_{\vk,\vv,\vec{1}} \cong \bar{V}_{\vk,\0,\vec{1}}
&\cong &
\ind(\omega_{2a}^{\sum_{j\in\Z/a\Z}k_j p^j})\otimes \bar\eta
\end{eqnarray*}
is irreducible.
\end{proposition}
\begin{proof}
We note that our first congruence follows from Theorem \ref{T:family}.
Let $N$ be the integral Wach module of $V_{\vk,\0,\1}$
and let $\bar{N}=N/\m_E$ be its
mod $p$ reduction. We shall prove $\bar{N}$ is irreducible
as an \'etale $(\varphi,\Gamma)$-module over $\Fp((\pi))^{\rm sep}$.
Recall that
$$
\Mat(\varphi|_{\bar{N}})_j =
\left\{
\begin{array}{lll}
&\left(
                                 \begin{array}{cc}
                                   \pi^{(p-1)k_j} & 0 \\
                                   0              & 1 \\
                                 \end{array}
                               \right) &\mbox{if $j\in\cB$};\\
&\left(
                                 \begin{array}{cc}
                                   0  & -1 \\
                                   \pi^{(p-1)k_j} & 0 \\
                                 \end{array}
                               \right) &\mbox{if $j\in\cA$}.
\end{array}
\right.
$$
Then
$$\Mat(\varphi^a|_{\bar{N}}) = \left(
                                 \begin{array}{cc}
                                   0  & -\pi^{(p-1)\sum_{j=0}^{a-2}k_jp^j} \\
                                   \pi^{(p-1)k_{a-1}p^{a-1}}  & 0 \\
                                 \end{array}
                               \right).
$$
Suppose $\bar{N}$ contains a proper subobject $\bar{N}'$.
Then $\varphi^a$ acts on $\bar{N}'=x e_1+y e_2$ for some $x,y\in \Fp((\pi))$
by
$$
\left(
\begin{array}{cc}
0  & -\pi^{(p-1)\sum_{j=0}^{a-2}k_jp^j} \\
\pi^{(p-1)k_{a-1}p^{a-1}}  & 0 \\
\end{array}
\right)
\left(
  \begin{array}{c}
    x^{p^a} \\
    y^{p^a} \\
  \end{array}
\right)
=\lambda \left(
           \begin{array}{c}
             x \\
             y \\
           \end{array}
         \right)
$$
for some $\lambda\in\F_p[[\pi]]$.
From the equation we arrives at
$(x/y)^{p^a+1}=-\pi^{(p-1)(k_0+k_1p+\ldots-k_{a-1}p^{a-1})}$,
which has no solution to $x/y$ in $\Fp((\pi))$.
This proves irreducibility of $\bar{V}_{\vk,\0,\vec{1}}$.

Let $\bar{W}$ be the determinant of $\bar{N}$. We have that
$\Mat(\varphi|_{\bar{W}})=(\pi^{(p-1)k_0},\ldots,\pi^{(p-1)k_{a-1}})$.
It is clear that $\bar{W} = \D(\omega_a^{\sum_{j\in\Z/a\Z}k_jp^j})$
by comparison with Proposition \ref{P:omega_a}.

Finally we claim that
$$\bar{V}|_{I_K} \cong \omega_{2a}^{\sum_{j\in\Z/a\Z}k_jp^j}
\oplus \omega_{2a}^{p^a\sum_{j\in\Z/a\Z} k_jp^j}.$$
It is reduced to consider $\bar{N}$ over
$\tilde\E\otimes_{\F_{p^a}} {\F_{p^{2a}}}
\cong \prod_{i=0}^{1}\tilde{\E}$
to be isomorphic to
$\prod_{i=0}^{1}
\D(\omega_{2a}^{p^{ai}\sum_{j\in\Z/a\Z}k_jp^j})$.
Denote the basis adapted to filtration by
$\bar{N}_j=\pscal{e_{j1},e_{j2}}$ for every $j$.
Let
\begin{eqnarray*}
\vec{e}'_1:=(e_{0,1},e_{1,1},\ldots, e_{a-1,1}, e_{0,2},e_{1,2},\ldots,e_{a-1,2})\\
\vec{e}'_2:=(e_{0,2},e_{1,2},\ldots, e_{a-1,2}, e_{0,1},e_{1,1},\ldots,e_{a-1,1})
\end{eqnarray*}
Change basis to the matrix of $\varphi$ above to that
with respect to $\vec{e}'_1$ and $\vec{e}'_2$ we have
\begin{eqnarray*}
\Mat(\varphi|_{\vec{e}'_1})&=& (\pi^{(p-1)k_0}, \ldots,\pi^{(p-1)k_{a-1}}, 1,\ldots,1).\\
\Mat(\varphi|_{\vec{e}'_2})&=& (1,\ldots,1, \pi^{(p-1)k_0}, \ldots,\pi^{(p-1)k_{a-1}}).
\end{eqnarray*}
Then it is clear that
$\pscal{\vec{e}'_1} \cong  \D(\omega_{2a}^{\sum_{j\in\Z/a\Z}k_jp^j})$
and $\pscal{\vec{e}'_2} \cong  \D(\omega_{2a}^{p^a\sum_{j\in\Z/a\Z}k_jp^j})$.
This proves our claim about $\bar{V}_{I_K}$ above.
\end{proof}

We remark that our work in dimensional $2$ case (in particular, Section 5.2) overlaps with recent and current work of Dousmanis \cite{Do07,Do08}.
In particular, Dousmanis lists isomorphism classes
of 2-dimensional weakly admissible filtered $\varphi$-modules
with Galois descent data and coefficients.
His work allows him to describe the isomorphism classes of 2-dimensional potentially crystalline representations \cite[Section 4.1]{Do07} and potentially
semistable (noncrystalline) representations \cite[Section 4.2]{Do07}.
In \cite[Theorem 3]{Do08} Dousmanis has shown that if
$\ord_p v_j > \lfloor{(\max_{j\in\Z/a\Z}(k_j)-1)/(p-1)}\rfloor$
for all $j\in\Z/a\Z = \cA$, then $\bar{V}_{\vk,\vv} \cong \bar{V}_{\vk,\0}$.

\end{document}